\documentclass[11pt,reqno]{amsart}
\usepackage[english]{babel} 
\usepackage{color}
\usepackage{amsmath, amsthm, amssymb}
\usepackage{amsfonts}
\usepackage[ansinew]{inputenc}
\usepackage{graphicx}
\usepackage{babel}
\usepackage{enumerate}
\usepackage{hyperref}
\theoremstyle{plain}
\newtheorem{thm}{Theorem}[section]
\newtheorem{cor}[thm]{Corollary}
\newtheorem{lem}[thm]{Lemma}
\newtheorem{prop}[thm]{Proposition}

\theoremstyle{definition}
\newtheorem{defi}[thm]{Definition}

\theoremstyle{remark}
\newtheorem{rem}[thm]{Remark}
	 
\numberwithin{equation}{section}

\newcommand{\de}{\partial}

\newcommand{\R}{\mathbb{R}}

\newcommand{\N}{\mathbb{N}}

\newcommand{\eps}{\varepsilon}
\newcommand{\HH}{\mathcal H}

\newcommand{\average}{{\mathchoice {\kern1ex\vcenter{\hrule height.4pt
width 6pt depth0pt} \kern-9.7pt} {\kern1ex\vcenter{\hrule
height.4pt width 4.3pt depth0pt} \kern-7pt} {} {} }}

\def\R{\mathbb{R}}

\begin{document}

\title{Regularity of minimal surfaces with lower dimensional obstacles}

\author{Xavier Fern\'andez-Real}

\address{ETH Z\"urich, Department of Mathematics, R\"amistrasse 101, 8092 Z\"urich, Switzerland}

\email{xavierfe@math.ethz.ch}

\author{Joaquim Serra}

\address{ETH Z\"urich, Department of Mathematics, R\"amistrasse 101, 8092 Z\"urich, Switzerland}

\email{joaquim.serra@math.ethz.ch}



\begin{abstract} 
We study the Plateau problem with a lower dimensional obstacle in $\mathbb{R}^n$. Intuitively, in $\mathbb{R}^3$ this corresponds to a soap film (spanning a given contour) that is pushed from below by  a ``vertical'' 2D half-space (or some smooth deformation of it). 
We establish almost optimal $C^{1,1/2-}$ estimates for the solutions near points on the free boundary of the contact set, in any dimension $n\ge 2$. 

The $C^{1,1/2-}$ estimates follow from an $\varepsilon$-regularity result for minimal surfaces with thin obstacles in the spirit of the De Giorgi's improvement of flatness.  To prove it, we follow Savin's small perturbations method. A nontrivial difficulty in using Savin's approach for minimal surfaces with thin obstacles is that near a typical contact point the solution consists of two smooth surfaces that intersect transversally, and hence it is not very flat at small scales.
Via a new  ``dichotomy approach'' based on barrier arguments we are able to overcome this difficulty and prove the desired result.
\end{abstract}

\maketitle

\section{Introduction}

\subsection{Minimal surfaces with obstacles}

In this paper we study the regularity of minimizers in the Plateau problem with a lower dimensional --- or {\em thin} --- obstacle. Before introducing the problem in further detail let us contextualize it by recalling five closely related classical problems and commenting on them.
\begin{itemize}
\item The Plateau problem:
\begin{equation}\label{eq.plateuspb}
 \min \big\{ P(E; B_1) \ :\  E\setminus B_1 = E_\circ \setminus B_1  \big\},
\end{equation}
where $E_\circ\subset \R^n$ (boundary condition), and  $B_1$ denotes the unit ball of $\R^{n}$, $E\subset \R^n$,  and $P(E;B_1)$ denotes the relative perimeter of the set $E$ in $B_1$.

\item The Plateau problem with an obstacle:
\begin{equation}\label{eq.thepb}
 \min \big\{ P(E; B_1 ) \ :\  E \supset \mathcal O, \  E\setminus B_1 = E_\circ \setminus B_1   \big\}\,
\end{equation}
where $E_\circ,E$ are as above and  $\mathcal O\subset E_\circ$ (the obstacle) is given. 

\item The nonparametric obstacle problem:
\begin{equation}\label{eq.OPnonpar}
\min_{v}\biggl\{\int_{B_1'} \sqrt{1+|\nabla v|^2}\,:\, v \geq \psi \text{ in $B_1'$},\,v|_{\partial B_1'}=g\biggr\},
\end{equation}
where $B_1'$ denotes the unit ball of $\R^{n-1}$,  $g:\partial B_1'\to \R$ (the boundary condition) is given, $v:B_1'\to \R$,  and $\psi:B_1'\to \R$ is the obstacle
satisfying $\psi|_{\partial B_1'}<g$.

\item The obstacle problem:
\begin{equation}\label{eq.OP}
\min_{v}\biggl\{\int_{B_1'} \frac{|\nabla v|^2}{2}\,:\, v \geq \psi \text{ in $B_1'$},\,v|_{\partial B_1'}=g\biggr\},
\end{equation}
where $g$, $v$, and $\psi$, are as above.

\item The Signorini problem, or {\rm thin} obstacle problem:
\begin{equation}\label{eq.Sign}
\min_{v}\biggl\{\int_{B_1'} \frac{|\nabla v|^2}{2}\,:\, v \geq \psi \text{ in $B_1'\cap\{x_{n-1}=0\}$},\,v|_{\partial B_1'}=g\biggr\},
\end{equation}
where $g$ and $v$ are as above, and now $\psi:B_1'\cap\{x_{n-1}=0\}\to \R$ (the thin obstacle) acts only on $\{x_{n-1}=0\}$.
 \end{itemize}
 
Note that \eqref{eq.OPnonpar} is a particular case of \eqref{eq.thepb}, namely, when $\partial \mathcal O$ and $\partial E$ are  graphs. Also, \eqref{eq.OP} is, in turn,  a limiting case of \eqref{eq.OPnonpar} --- for $\eps$-flat graphs, the area functional $\int \sqrt{1+|\eps \nabla v|^2 }$ becomes the Dirichlet energy $\int \frac 12 |\eps \nabla v|^2$ at leading order.

The regularity of solutions and free boundaries is nowadays well understood in both the classical obstacle problem \eqref{eq.OP} --- see  \cite{Caf77, Caf98} --- and in the Signorini problem --- see \cite{AC04,ACS08}.
The case of minimal surfaces with thick obstacles (both in parametric and nonparametric form) is also well understood  --- see  \cite{Kin73, BK, Jen80, Giu10}.

This paper is concerned with the regularity of minimizers of the Plateau problem with lower dimensional, or thin, obstacles. Namely, we consider \eqref{eq.thepb} with obstacle
\begin{equation} \label{eq.thinO}
\mathcal O:= \Phi \big(\{x_{n-1}=0, x_n\le0\}\big) 
\end{equation}
where $\Phi:\R^{n}\rightarrow \R^n$ is some smooth ($C^{1,1}$) diffeomorphism. We denote 
\[
\partial \mathcal O:= \Phi \big(\{x_{n-1}=0, x_n=0\}\big). 
\]
 
This problem \eqref{eq.thepb}-\eqref{eq.thinO} is the geometric version of the  Signorini problem \eqref{eq.Sign} in the same way that  \eqref{eq.thepb} with thick $\mathcal O$ is the geometric version of \eqref{eq.OP}.
To visualize a solution of this problem in $\R^3$, one can think of a soap film (spanning a given contour) that is pushed from below by a vertical 2D half-space, as depicted in Figure \ref{fig.potato}.
Note that, in $\R^3$, we cannot use a ``wire'' (i.e. a one dimensional curve) as obstacle, since the surface will not ``feel'' it\footnote{More precisely, one can see that if $\mathcal O$ had codimension two, then solutions of \eqref{eq.thepb} with an infinitesimal tubular neighbourhood of $\mathcal O$ as obstacle would become, in the limit, solutions of the Plateau problem \eqref{eq.plateuspb} (without obstacle).}.

\begin{figure*}[h]  
\centering
\includegraphics[width = 11cm]{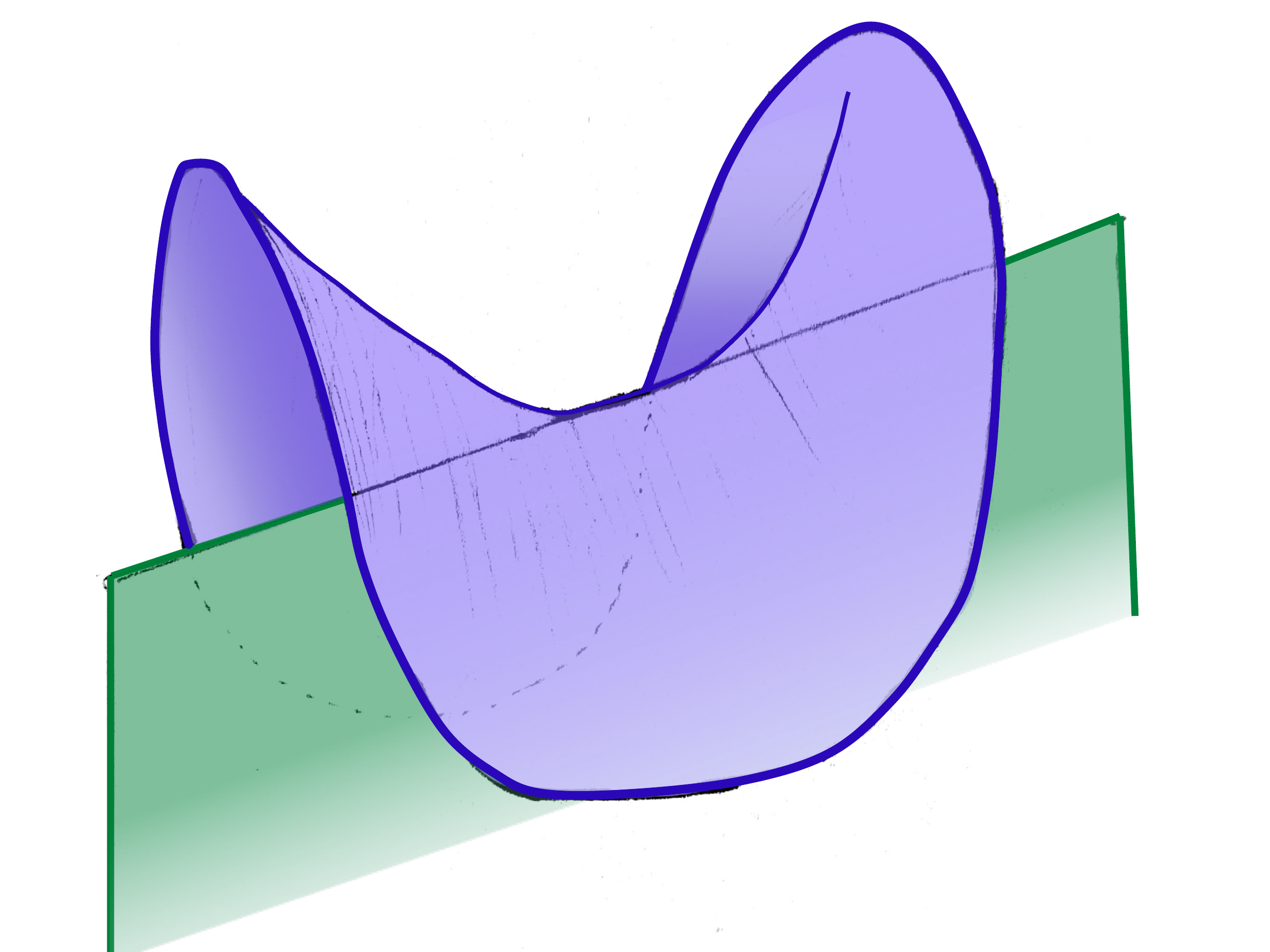}
\caption{The ``potato chip configuration'', popularized by Caffarelli.}
\label{fig.potato}
\end{figure*}

Although the problem of  minimal surfaces with thin obstacles was introduced by De Giorgi  \cite{DGthin} already in 1973  (he established an existence result), very little was known on the regularity of its solutions. De Acutis in \cite{Deacu} established $C^1$ regularity around points of the solution belonging to $\mathcal{O}\setminus \de\mathcal{O}$. To our knowledge, the only known regularity results up to $\de\mathcal{O}$ concern the nonparametric case --- as in \eqref{eq.OPnonpar} but with $\psi$ as in \eqref{eq.Sign}.  They are due to Kinderlehrer \cite{Kin71} who proved $C^1$ regularity estimates for the solution in two dimensions, and to Giusti  \cite{Giu72}, who obtained Lipschitz estimates for the solution in every dimension.

The difficulty in studying \eqref{eq.thepb}-\eqref{eq.thinO} (with respect to the same problem with a thick obstacle) lies on the fact that near a typical point of the contact set the hypersurface $\partial E$ consists of two surfaces that intersect transversally on $\partial \mathcal O$. Therefore, $\partial E$ is typically not flat at small scales and thus \eqref{eq.thepb} cannot be treated as a perturbation of \eqref{eq.Sign}.
A more subtle dichotomy argument is needed: in Subsection \ref{ssec.main} we outline the idea of this new approach that is tailored to overcome the previous difficulty.

Let us also point out  that it is not completely obvious  how to give  a meaningful notion of solution to \eqref{eq.thepb}-\eqref{eq.thinO}. The main issue is that with the Caccioppoli definition of relative perimeter $P$ we have
\begin{equation}\label{eq.warning}
P(E\cup \mathcal O;B_1) = P(E; B_1 ) \quad \mbox{for all measurable } E,
\end{equation}
and thus the obstacle $\mathcal O$ seems to be ignored by $P$. 
This issue led De Giorgi \cite{DGthin} to introduce a more appropriate notion of perimeter that is suitable for the study of thin obstacle problems (this is currently known as the De Giorgi measure).  
We choose the similar (and a posteriori equivalent) approach of  looking at the thin obstacle as a limit of infinitesimaly thick neighbourhoods of it. See Subsection \ref{ssec.defsol} for a more detailed discussion on this issue.

The goal of this paper is to address the question of  the regularity of solutions to \eqref{eq.thepb}-\eqref{eq.thinO}. In particular, the main result of this paper is the proof of the following local almost optimal regularity result.

\begin{thm}
\label{thm.generic}
Let $E$ be a solution to the thin obstacle problem \eqref{eq.thepb}-\eqref{eq.thinO} in the unit ball of $\R^n$, $n\ge2$. Then, $\de E$ is $C^{1,1/2-}$ around contact points and up to the contact set.
\end{thm}

The appropriate notion of solution is discussed in Subsection \ref{ssec.defsol}.
Let us emphasize here that this local regularity near contact points result holds in any dimension $n\ge 2$, in contrast to the classical regularity theory of minimal surfaces in which minimizers are regular only up to dimension $7$. As we will see, this difference is due to the presence of the thin obstacle, which rules out solutions with singularities of the type of Simons and Lawson's cones like those appearing in dimension $n\ge 8$ in the Plateau problem without obstacles.

In the following subsections we recall the main steps in the regularity theory for sets of minimal perimeter and present the appropriate analogues for \eqref{eq.thepb}-\eqref{eq.thinO}.

 \subsection{Improvement of flatness}
For the classical Plateau problem De Giorgi \cite{dG61} established, in 1961, the following fundamental result:
\begin{thm}[\cite{dG61}]\label{eq.flatimpliessmooth}
Let $E\subset \R^n$ be a minimizer of the perimeter functional in $B_1$ and assume that $\partial E\cap B_1 \subset  \{ |e\cdot x|\le \eps_\circ \}$ for some $e\in \mathbb S^{n-1}$, where $\eps_\circ=\eps_\circ(n)$ is some positive dimensional constant. Then, 
$\partial E\cap B_{1/2}$ is a smooth hypersurface.
\end{thm}

This theorem follows from the following \emph{improvement of flatness} property for minimizers $E$ of the perimeter in $B_1$. Namely, given $\alpha\in (0,1)$ there exist positive constants $\eps_\circ (n,\alpha)$ and $\rho_\circ(n,\alpha)$ such that, whenever $0\in \partial E$ and $\eps\in(0,\eps_\circ)$ then the following implication holds:
\begin{equation}\label{eq.improofflatness}
 \partial E\cap B_1 \subset  \big\{ |e\cdot x|\le \eps \big\}  \quad    \Rightarrow \quad  \partial E\cap B_{\rho_\circ } \subset  \left\{  |\tilde e \cdot x|\le \eps\rho_\circ^{1+\alpha} \right\}.
\end{equation}
Here, $e$ and $\tilde e$ denote two possibly different unit vectors (in $\mathbb S^{n-1}$). 

Combined with the classification of stable minimal cones by Simons \cite{Sim68}, Theorem \ref{eq.flatimpliessmooth} yields that minimizers of the perimeter in $\R^n$ are smooth for $3\le n\le 7$. This result is optimal since, in dimensions $n\ge 8$, Bombieri, De Giorgi, and Giusti \cite{BdGG69} showed the existence of minimal boundaries with an $(n-8)$-dimensional linear space of cone-like singularities.

The philosophy of Theorem \ref{eq.flatimpliessmooth} is also shared by other key regularity results of nonlinear PDEs: \emph{if a solution happens to be close enough to some special solution (e.g., the hyperplane), then it is regular}. 
These are the so-called ``$\eps$-regularity results''.

The goal of the  paper is  to establish an $\eps$-regularity result for \eqref{eq.thepb}-\eqref{eq.thinO}, thus extending De Giorgi's improvement of flatness theorem to the setting of problem \eqref{eq.thepb}-\eqref{eq.thinO} --- see Theorem \ref{thm.main} below. 
As a consequence, we will prove almost optimal $C^{1,1/2-}$ estimates for minimizers of \eqref{eq.thepb}-\eqref{eq.thinO} in $\R^n$ that are sufficiently close to a canonical blow-up solution (the {\em wedges} introduced in the following subsection).
We will also see that these canonical blow-up solutions are the only possible blow-ups at any contact point, and then Theorem \ref{thm.generic} will follow.

\subsection{Blow-ups} 
\label{ssec.blowup}
An essential tool in the theory of minimal surfaces is the monotonicity formula. Namely, if $\partial E$ is a minimal surface and $x_\circ\in \partial E$, then the function
\begin{equation}\label{eq.phi}
\mathcal A(r) := \frac{1}{r^{n-1}}\HH^{n-1}\big(\partial E\cap B_r(x_\circ)\big)
\end{equation}
is monotone nondecreasing. In addition, $\mathcal A$ is constant if and only if $E$ is a cone. 
A standard consequence of this monotonicity formula is that blow-ups of a minimizer of the perimeter $E\subset \R^n$ at any point $x_\circ\in \partial E$ are {\em minimizing cones}. 
Simons proved in \cite{Sim68} that half-spaces are the only minimizing cones in dimensions $n\le 7$.
As a consequence,  one can always apply Theorem \ref{eq.flatimpliessmooth} near $x_\circ$ after zooming in enough --- this gives the smoothness of perimeter minimizers for $n\le 7$.

For problem \eqref{eq.thepb}-\eqref{eq.thinO} we find several analogies with this theory. As we will prove in Lemma \ref{lem.mono}, if $E$ is a minimizer of \eqref{eq.thepb}-\eqref{eq.thinO} and $x_\circ \in \partial E\cap \partial \mathcal O$ is a contact point, then the same function $\mathcal A(r)$ in \eqref{eq.phi} is still monotone when $\Phi = {\rm id}$ (and an approximate monotonicity formula is also available for general smooth $\Phi$; see Lemma \ref{lem.mono}). As a consequence, blow-ups are also cones for \eqref{eq.thepb}-\eqref{eq.thinO}.
It is trivially false, however, that hyperplanes are the only possible blow-ups in low dimensions.
Indeed, the {\em wedges} (see Figure~\ref{fig.lambda0})
\begin{equation}
\label{eq.deflambda0}
 \Lambda_{\gamma, \theta} := \big\{x\in \R^n  \,:\, e_{\gamma+\theta}\cdot x\le 0 \  \mbox{  and  }\ e_{\gamma-\theta} \cdot x\le 0\big\},
\end{equation}
for
\begin{equation}\label{eq.eandtheta}
e_\omega := \sin\omega\, \boldsymbol e_{n-1} +\cos\omega\, \boldsymbol e_{n}, 
 \qquad   -\frac \pi 2 \le \gamma \le \frac \pi 2, \qquad  0\le \theta \le \frac \pi 2-  |\gamma|.
\end{equation}
are solutions to \eqref{eq.thepb}-\eqref{eq.thinO} for $\Phi = {\rm id}$ . Thus, they are always possible blow-ups.

Being a wedge, $\Lambda_{{\gamma}, \theta}$ is the intersection of two semispaces with normal vectors contained in the plane generated by $\boldsymbol{e}_{n-1}$ and $\boldsymbol{e}_{n}$. The aperture angle of the wedge is given by $\pi - 2\theta$, while its rotation angle is given by $\gamma$ with respect to $\boldsymbol e_{n}$ (we take the convention that $\boldsymbol{e}_{n-1} = e_{\pi/2}$). Note also that there is the restriction $0\le \theta \le \frac{\pi}{2}-|\gamma|$ to guarantee that the obstacle $\{x_{n-1}= 0, x_n\le 0\}$ is contained in $\Lambda_{{\gamma}, \theta}$. 

\begin{figure*}[t]
\centering
\includegraphics[width = 15cm]{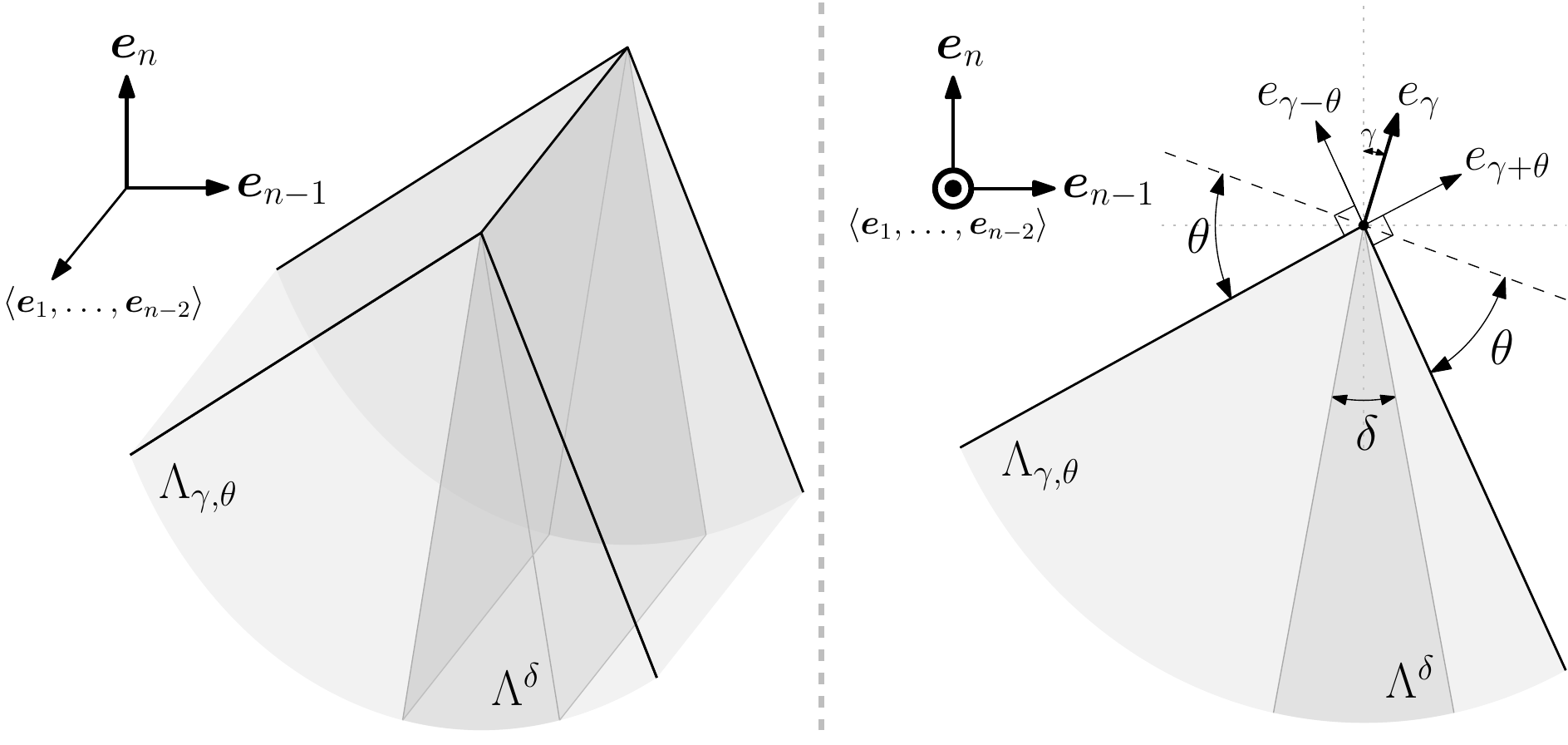}
\caption{Representations of $\Lambda_{{\gamma}, \theta}$ and $\Lambda^{\delta} $.}
\label{fig.lambda0}
\end{figure*}

We will show that, in all dimensions, the wedges are the only possible blow-ups around contact points.
More precisely, if $E$ is a minimizer of \eqref{eq.thepb}-\eqref{eq.thinO} and $x_\circ\in \partial E\cap \partial  \mathcal O$ (i.e. $x_\circ$ is a contact point) we have, in a suitable frame depending on $x_\circ$,
\begin{equation}\label{eq.blowupO}
\frac{1}{r_k}\big(\mathcal O-x_\circ\big)\  \longrightarrow \  \{x_{n-1} =0, \ x_n\le 0\} 
\end{equation}
and
\begin{equation}\label{eq.blowupE}
\frac{1}{r_k} \big(E-x_\circ\big)\  \longrightarrow  \Lambda_{{\gamma}, \theta}.
\end{equation}
This will be a consequence of the the classification of conic solutions to the thin obstacle problem, given in Proposition~\ref{prop.classif_cones}.

\subsection{Rigorous notion of solution to \eqref{eq.thepb}-\eqref{eq.thinO}} 
\label{ssec.defsol}
Given a measurable set $E$ and an open set $\Omega\subset \R^{n}$, we recall the standard definition of the relative perimeter of $E$ in $\Omega$ as 
\begin{equation}
\label{eq.min}
P(E;\Omega) = \int_{\Omega}|\nabla \chi_E| = \sup_{g\in C^1_0(\Omega), \|g\|_{L^\infty}\leq 1} \left|\int_{E}{\rm div }\,g\,\right|.
\end{equation}

With this definition of perimeter \eqref{eq.warning} holds. Thus, unless we define the problem with further precision, minimizers of \eqref{eq.thepb}-\eqref{eq.thinO} will be --- strictly speaking --- just the ones of \eqref{eq.plateuspb}, ignoring $\mathcal O$.

This, of course, is not what we have in mind when we think of \eqref{eq.thepb}-\eqref{eq.thinO}. Heuristically, we would like that if $\partial E$ attaches from both sides to $\mathcal O$ in some region, then the area of it is counted twice in the computation of the perimeter of $E$ instead of being ignored.
To solve this issue De Giorgi introduced in \cite{DGthin} a notion of perimeter that is suitable for the study of thin obstacle problems (the De Giorgi measure); see also \cite{Deacu}.  
Here we will use the similar approach (that will be a posteriori equivalent) of considering a thin obstacle as a limit of thick obstacles.

Let us introduce the precise notion of \eqref{eq.thepb}-\eqref{eq.thinO} that will be used in this paper. 
For $\delta>0$ small, let us denote 
\begin{equation}
\label{eq.deflambdad}
\Lambda^{\delta} := \Lambda_{0, \frac \pi 2 -\delta}.
\end{equation}
(Note that $\Lambda^{\delta}$ is very sharp wedge, pointing in the $\boldsymbol{e}_{n}$ direction.)

\begin{defi}\label{defi.notionmin}
We say that $E$ is a \emph{minimizer of \eqref{eq.thepb}-\eqref{eq.thinO}} in $B_1$ if $E$ has positive density at some point of $\mathcal O$  and there exist $\delta_k\downarrow 0$, $E_k$ minimizers of 
\begin{equation}\label{eq.wedgepb}
 \min \Big\{ P(\tilde E; B_1 ) \ :\  \tilde E\setminus B_1 = \big(E_\circ \cup  \Phi(\Lambda^{\delta_k}) \big)\setminus B_1 \quad \mbox{and}\quad  \Phi(\Lambda^{\delta_k}) \subset \tilde E  \Big\}
\end{equation}
such that $\chi_{E_k}\rightarrow \chi_E$ in $L^1(B_1)$. 
\end{defi}
Note that $\Phi\big(\Lambda^{\delta_k}\big)$ are \emph{thick} sets approximating $\mathcal O$. 
Now, minimizers of \eqref{eq.wedgepb} ``feel'' the obstacle no matter how small $\delta_k$ is. 
The intuitive idea behind this definition is that a sequence $E_k$ as in Definition \ref{defi.notionmin} will not converge to a solution to the Plateau problem unless the obstacle $\mathcal O$ is ``inactive'' (i.e., the obstacle is contained in density one points for the solution to the Plateau problem). 
The philosophy of the paper will be to prove regularity estimates for problem \eqref{eq.wedgepb} that are robust as $\delta_k\downarrow 0$. As a consequence, we will be able to show that the previous intuitive idea is actually fact. Namely, as it will be clear from the results of the paper, if the solution to the Plateau problem (with boundary data $E_\circ$) crosses $\mathcal O\setminus \de\mathcal{O}$, then there exists a minimizer of \eqref{eq.thepb}-\eqref{eq.thinO} which is not a solution of Plateau problem (and therefore, the thin obstacle plays an active role).

We remark that any minimizer according to Definition~\ref{defi.notionmin} (up to replacing  the complement of $E$ by the zero density points of $E$) is a minimizer in the sense of De Giorgi by \cite{Deacu} (see Remark~\ref{rem.vanishes}). Conversely, it is not true a priori that any minimizer in the sense of De Giorgi can be recovered as a minimizer in the sense of Definition~\ref{defi.notionmin}. Nonetheless, minimizers of the De Giorgi perimeter present \emph{locally} an aperture around the obstacle by \cite{Deacu} (and thus, a wedge fits within), and therefore, locally around contact points they are minimizers in the sense of Definition~\ref{defi.notionmin}.  In particular, since our regularity results are local, they apply to minimizers in the sense of De Giorgi. (See Remark~\ref{rem.dg}.)

\subsection{Regularity for solutions sufficiently close to a wedge}\label{ssec.main}

The first result of this paper is stated next, after introducing some notation and a definition. Throughout the paper we will denote
\[ X\subset Y \mbox{ in }B  \quad\Leftrightarrow\quad  X\cap B\subset Y\cap B.\]
We also introduce the following 
\begin{defi}\label{defi.epsclose} 
We say that $E$ is $\eps$-{\em close} to $\Lambda_{{\gamma},\theta}$ in $B$ if
\[
 \Lambda_{{\gamma},\theta}^{-\eps} \subset E\subset\Lambda_{{\gamma},\theta} ^{\eps} \quad  \mbox{in }  B
 \]
 where 
 \[\Lambda_{{\gamma},\theta}^\eps := \{x\in \R^n \, : \, {\rm dist}(x,\Lambda_{{\gamma},\theta}) \le \eps \}, \quad \Lambda_{{\gamma},\theta}^{-\eps} := \{x\in \R^n \, : \, {\rm dist}(x,\R^n\setminus\Lambda_{{\gamma},\theta}) \ge \eps \}.\]
\end{defi}

Here is our main result, which we call {\em improvement of closeness}:
\begin{thm}[Improvement of closeness]\label{thm.main}
Given $\alpha\in\left(0,\frac 1 2\right)$ there exist positive constants $\eps_\circ$ and $\rho_\circ$ depending only on $n$ and $\alpha$ such that the following holds:

Assume that, for some $\delta>0$, a set $E\subset \R^n$ with $P(E;B_1)<\infty$ satisfies $\Phi(\Lambda^{\delta}) \cap B_1\subset E$ and
\begin{equation}\label{eq.TOP}
 P(E;B_1)\le P(F;B_1) \quad \forall F\  \mbox{such that }E\setminus B_1 = F\setminus B_1 \mbox{ and } \Phi(\Lambda^{\delta}) \cap B_1\subset F. 
\end{equation}
Suppose that $0\in \partial E\cap \partial \mathcal O$, $\eps \in(0,\eps_\circ)$, and
\begin{equation}
\label{eq.Phihyp}
\Phi(0)=0,\quad D\Phi(0) = {\rm id},\quad|D^2\Phi|\le \eps^{1+\frac 1 2}.
\end{equation}

Then, 
\begin{equation}\label{eq.improvement}
\mbox{$E$ is $\eps$-close to $ \Lambda_{{\gamma},\theta}$  in  $B_1$}   \quad \Rightarrow \quad  \mbox{$E$ is  $\eps\rho_\circ^{1+\alpha}$-close to  $\Lambda_{\tilde \gamma, \tilde \theta}$  in  $B_{\rho_\circ}$},
\end{equation}
where $\gamma$, $\tilde \gamma$, $\theta$, and $\tilde \theta$, are as in \eqref{eq.eandtheta}.
\end{thm}

\begin{rem} Let us comment on the statement of Theorem \ref{thm.main}:
\begin{enumerate}
\item[(1)] This result generalizes the classical De Giorgi's improvement of flatness theorem \eqref{eq.improofflatness}.

\item[(2)] Our estimate \eqref{eq.improvement} is designed to be applied, iteratively in a sequence of dyadic balls, to a minimizer $E$ of \eqref{eq.wedgepb}. It gives $C^{1,\alpha}$ regularity of $\partial E$ at points of the contact set; see Theorem~\ref{thm.main2} below.
\item[(3)] An essential feature of our result is that the constant $\eps_\circ$ is independent of $\delta$. Thus \eqref{eq.improvement} is stable as $\delta\downarrow 0$ and hence applies to solutions of \eqref{eq.thepb}-\eqref{eq.thinO}; see Definition \ref{defi.notionmin}.
\item[(4)]The assumption $\alpha<1/2$ is almost sharp. Indeed, one can easily see that the statement of the theorem cannot be true for $\alpha\in(\frac 1 2,1)$ by using that the optimal regularity of solutions to the Signorini problem is $C^{1,\frac 1 2}$.
\item[(5)] If $\Phi:\R^n \rightarrow \R^n$ is any $C^{1,1}$ diffeomorphism and $x_\circ$ belongs to $\partial \mathcal O = \Phi(\{ x_{n-1}= x_n=0\})$, then for $\rho>0$ and in some new coordinates $\bar x =\psi_{x_\circ}(x)$ with origin at $x_\circ$ such that
\[ \psi_{x_\circ}(x) :=  \rho^{-1}R_{x_\circ} (x -x_\circ), \quad \mbox{where $R_{x_\circ}$ is an orthogonal matrix}, \]
the assumption \eqref{eq.Phihyp} will be fulfiled by some new diffeomorphism $\bar \Phi$ satisfying $\bar\Phi(\Lambda^{\bar\delta}) = \psi(\Phi(\Lambda^\delta))$ --- see Lemma \ref{lem.asumpPhi}. Hence, assumption \eqref{eq.Phihyp} is always satisfied after a change of coordinates.
\end{enumerate}
\end{rem}

\subsection{On the proof of Theorem~\ref{thm.main}}\label{ssec.main2}

Let us now briefly comment on the proof of Theorem \ref{thm.main}. Our main idea is to use a ``dichotomy approach'', which is combined with Savin's ``small perturbation method''.
More precisely, we prove by a barrier argument that --- if $\eps_\circ$ is small enough --- one of the following two alternatives must hold:
\begin{enumerate} 
\item [$(a)$] $\partial E$ is very flat in $B_1$.
\item [$(b)$] The contact set is full in $B_{3/4}$ (it contains $\partial \mathcal O \cap B_{3/4}$) and $\partial E$ splits into two minimal surfaces that meet along $\partial\mathcal O$ with some angle.
\end{enumerate}
Then, on the one hand, if $(a)$ holds we can use that our problem is a perturbation of the Signorini problem \eqref{eq.Sign} and exploit the $C^{1,1/2}$ regularity for \eqref{eq.Sign} to prove \eqref{eq.improvement}. For this we use the ``small perturbation method'' pioneered by Savin --- see \cite{Sav09, Sav10, Sav10b}.

On the other hand, if $(b)$ holds then $\partial E$ splits in $B_{3/4}$ into two minimal surfaces with boundary, each of them flat in a different direction. Since the contact set is full we can interpret it as a smooth ``boundary condition''. Then, using the $C^{1,1}$ regularity up to the boundary of flat minimal surfaces, we can improve the flatness of each of the two surfaces separately to prove \eqref{eq.improvement}.

\subsection{Consequences}
From our Theorem~\ref{thm.main}, as in the classical theory, we get that once the minimizer is sufficiently close to a ``wedge'' type set $\Lambda_{{\gamma},\theta}$, then it has a local $C^{1,\alpha}$ structure. 

\begin{thm}\label{thm.main2}
Given $\alpha\in\left(0,\frac 1 2\right)$ there exists a positive constant $\eps_\circ$ depending only on $n$ and $\alpha$ such that the following holds: 

Assume that, for some $\delta>0$, a set $E\subset \R^n$ with $P(E;B_1)<\infty$ satisfies $\Phi(\Lambda^{\delta}) \cap B_1\subset E$ and \eqref{eq.TOP}.
Suppose that $0\in \partial E\cap \partial \mathcal O$, that
\begin{equation}
\label{eq.Phihyp2}
\Phi(0)=0,\quad D\Phi(0) = {\rm id},\quad|D^2\Phi|\le \eps_\circ^{1+\frac 1 2},
\end{equation}
and that $E$ is $\eps_\circ$-close to $\Lambda_{\gamma,\theta}$ in $B_1$.

Then, 
 $\de E$ has the following $C^{1,\alpha}$ structure in $B_{1/2}$. Either:
\begin{enumerate}
\item[(a)] In some appropriate coordinates $y = (y',y_n) = (y_1,\dots,y_n)$, $\Phi^{-1}(\partial E)$ is the graph $ \{y_n = h(y')\}$ of a function $h\in C^0(\overline{B_{1/2}'})$ that belongs to $C^{1,\alpha}(\overline{B_{1/2}'^+})\cap C^{1,\alpha}(\overline{B_{1/2}'^-})$, where $B_{1/2}'$ denotes the ball in $\R^{n-1}$ and $B_{1/2}'^\pm$ are the half-balls $B_{1/2}'\cap \{\pm y_{n-1} > 0\}$. Moreover, we have $h\ge 0$ on $y_{n-1} = 0$ and $\nabla h$ is continuous on $\{y_{n-1} = 0\}\cap \{h > 0\}$. 
\end{enumerate}
\noindent or 
\begin{enumerate}
\item[(b)]   $\partial E \cap B_{1/2}$ is the union of two $C^{1,1-}$ surfaces that meet on $\partial \mathcal O$ with full contact set in $B_{1/2}$.
\end{enumerate}
\end{thm}
In the previous statement $C^{1,1-} :=\bigcap_{\beta\in(0,1)} C^{1,\beta}$.

\begin{rem}
\label{rem.minsurf}
It will be clear from the proofs that if $\mathcal O$ is a minimal surface (with boundary), then $\partial E$ cannot stick to $\mathcal O\setminus \partial \mathcal O$ and (b) must hold with the same regularity as that of $\partial O$. Namely, if $\partial \mathcal O$ is a $C^{k,\beta}$ (resp. analytic) codimension two surface, then the two surfaces in (b) will also be $C^{k,\beta}$ (resp. analytic), and not just $C^{1,1-}$.
\end{rem}

Theorem~\ref{thm.main2} requires the solution to be sufficiently close to a wedge-type set $\Lambda_{{\gamma}, \theta}$. Thanks to the following classification of global conical solutions to our problem, we will have that this is always the case (after rescaling) near any contact point. 

\begin{prop}[Classification of minimal cones in $\R^n$] 
\label{prop.classif_cones}
Let $\Sigma \subset \R^n$ be a cone, i.e. $t\Sigma = \Sigma$ for all $t > 0$, with $\partial\Sigma \neq \varnothing$. Suppose that $\Sigma$ satisfies \eqref{eq.TOP} with $\Phi\equiv {\rm id}$. 

Then, $\Sigma = \Lambda_{{\gamma}, \theta}$ for some $\gamma$ and $\theta$ as in \eqref{eq.eandtheta}.
\end{prop}

As a direct consequence of the combination of Theorem~\ref{thm.main2} and Proposition~\ref{prop.classif_cones} we obtain the following result (which is just a more precise version of Theorem~\ref{thm.generic} above), 
\begin{cor}\label{cor.cor_2}
Let  $n\ge 2$, and assume that $\mathcal O$ is a minimal surface and that  $\Phi \in C^{k,\beta}$ for some $k\ge 2$ and $\beta\in (0,1)$ --- or equivalently $\partial \mathcal O$ is of class $C^{k,\beta}$. 

Let $E$ be a solution (in the sense of Definition \ref{defi.notionmin}) of \eqref{eq.thepb}-\eqref{eq.thinO} with $x_\circ \in \de E\cap \de\mathcal O\cap B_{1/2}$. Then, for all $\alpha\in\left(0,\frac 1 2\right)$,
$\partial E$ has the following $C^{1,\alpha}$ local structure near $x_\circ$. For $r>0$ small enough, we have either:

\begin{enumerate}
\item[(a)]   In some appropriate coordinates $y = (y',y_n) = (y_1,\dots,y_n)$, $\Phi^{-1}(\partial E)$ is the graph $ \{y_n = h(y')\}$ of a function $h\in C^0(\overline{B_{r}'})$ that belongs to $C^{1,\alpha}(\overline{B_{r}'^+})\cap C^{1,\alpha}(\overline{B_{r}'^-})$, where $B_{r}'$ denotes the ball in $\R^{n-1}$ and $B_{r}'^\pm$ are the half-balls $B_{r}'\cap \{\pm y_{n-1} > 0\}$. Moreover, we have $h\ge 0$ on $y_{n-1} = 0$ and $\nabla h$ is continuous on $\{y_{n-1} = 0\}\cap \{h > 0\}$. 
\end{enumerate}
\noindent or 
\begin{enumerate}
\item[(b)]   $\partial E \cap B_r(x_\circ)$ is the union of two $C^{k,\beta}$ minimal surfaces with boundary that meet on $\partial \mathcal O$ with full contact set in $B_{r}(x_\circ)$.
\end{enumerate}
\end{cor}

\begin{rem}
\label{rem.dg}
By \cite[Theorem 2.1 and Theorem 2.2]{Deacu} (or by a standard barrier argument similar to that used in Hopf's lemma) if one considers a minimizer of the De Giorgi measure for obstacles as in Corollary \ref{cor.cor_2}, then its boundaries do not stick to 
the obstacle. More precisely, they present an aperture around the obstacle that allows, locally, a wedge contained in the minimizer.

As a consequence, minimizers of the  De Giorgi measure are
 locally (in a neighborhood of any contact point) minimizers in the sense of  Definition~\ref{defi.notionmin}.  Therefore, Corollary~\ref{cor.cor_2} above applies to minimizers in the sense of De Giorgi. 
\end{rem}

\begin{rem}
\label{rem.cor_2}
In the previous statement the condition that $\mathcal{O}$ is a minimal surface appears only to be able to apply Remark~\ref{rem.minsurf} and obtain (b). Otherwise, an analogous result with $C^{1,1-}$ regularity holds. 
\end{rem}

\begin{rem}
\label{rem.vanishes}
We observe that,  as a consequence of our results,
\begin{equation}\label{aioahiofa}
\mbox{$E$ is a minimizer as in Definition \ref{defi.notionmin}}\quad \Rightarrow \quad   P_{DG}(E;B_1) =P(E;B_1).
\end{equation}

Indeed, let $E$ be a minimizer as in Definition \ref{defi.notionmin}. First,  as proven in   \cite{Deacu}, since $\mathcal O$ is smooth, the De Giorgi perimeter $P_{DG}$ of the minimizer can be expressed as
\begin{equation}\label{hasiofhwaoih}
P_{DG}(F; B_1) = P(F; B_1) + 2\mathcal{H}^{n-1}((\mathcal{O}\setminus F)\cap B_1) \ge P(F; B_1)\quad \mbox{for any Borel set  }F.
\end{equation}

But note that $\partial E$ cannot stick to the obstacle from both sides at any point of $\mathcal O\setminus \partial \mathcal O$ by the strong maximum principle. Hence, 
\begin{equation}\label{hasiofhwaoih0}
\mathcal{H}^{n-1}((\mathcal{O}\setminus E)\cap B_1)=0.
\end{equation}
Using \eqref{hasiofhwaoih} and \eqref{hasiofhwaoih0}, $E$ is therefore also a minimizer of $P_{DG}$, since $P_{DG}(F; B_1)\ge P(F; B_1)\ge P(E; B_1)=P_{DG}(E;B_1)$ for any competitor $F$.

\end{rem}

\begin{rem}
Corollary~\ref{cor.cor_2} gives the regularity of the hypersurface around contact points. The regularity around other points follows from the classical theory for minimal surfaces (see for instance chapters 8 and 9 of the classical book of Giusti \cite{Giu84}). Note that this is result only up to dimension 7  \cite{Sim68} since nonsmooth minimizers exist in dimensions 8 and higher \cite{BdGG69}. In contrast, our regularity result holds around the contact set of the thin obstacle, in any dimension. 
\end{rem}

\begin{rem}
After a previous version of this manuscript, a preprint of Focardi and Spadaro \cite{FS} appeared in which the authors establish optimal $C^{1,1/2}$ regularity estimates and rectifiability of the free boundary for minimal surfaces with flat thin obstacles in the nonparametric case (that is, in our notation, for the case $\Phi = {\rm id}$ and  assuming that $\partial E$ is a graph in the $n$-th direction).
Interestingly, our Corollary \eqref{cor.cor_2} gives that (at least for flat obstacles) the assumptions of \cite{FS} are always satisfied near any contact point by parametric minimal surfaces with thin obstacles. Thus, when combined with our results, the results in \cite{FS} yield that solutions to parametric thin obstacle problems are $C^{1,1/2}$ near the obstacle and their free boundary is rectifiable.
\end{rem}

\subsection{Organization of the paper}

The paper is organised as follows. 

In Section~\ref{sec.2} we introduce some notation, definitions, and preliminary results. In Section~\ref{sec.3} we construct a barrier and prove the dichotomy presented in the introduction: if the solution is close to a wedge, then either $\de E$ is very flat or its contact set is full in a smaller ball. In Section~\ref{sec.4} we focus on the flat configuration, showing the improvement of closeness result in this case (Proposition~\ref{prop.impclos}). In Section~\ref{sec.5}, instead, we focus on the full contact set configuration, which allows us to complete the proof of our first main result, Theorem~\ref{thm.main}. In Section~\ref{sec.6} we prove Theorem~\ref{thm.main2} by iteratively applying Theorem~\ref{thm.main}. Finally, in Section~\ref{sec.7} we discuss blow-ups (monotonicity formula and  classification of minimal cones) and we complete the proofs of  Proposition~\ref{prop.classif_cones} and Corollary~\ref{cor.cor_2}, thus obtaining Theorem~\ref{thm.generic}.
\\

{\it Acknowledgement 1:} This work has received funding from the European Research Council (ERC) under the Grant Agreement No 721675. The second author was also supported by the Swiss National Science Foundation (Ambizione grant PZ00P2\_180042).\\

{\it Acknowledgement 2:}  we thank  M. Focardi, G. De Philippis, and E. Spadaro, for interesting discussions on the topic of the paper. We would also like to deeply thank Connor Mooney for pointing out to us the key observation that yields the classification of  minimal cones in every dimension (see Proposition~\ref{prop.classif_cones}).

\section{Notation and preliminary results}
\label{sec.2}
\subsection{Conventions and notation.}
As it is standard, throughout the paper we will assume that the representative of $E$ among sets that differ from it by a null set is such that topological and measure theoretic boundary agree. 
That is, given a set $E \subset \R^n$, we will say that $x\in \R^n$ belongs to the boundary of $E$, $x\in \de E$, whenever 
\[
0<|E\cap B_r(x) |< |B_r(x)|,\quad\mbox{ for all } r > 0.
\]

Notice that, in general, this is not necessarily true. However, the set of points where this does not hold is of measure zero, and therefore we can consider instead the equivalent set $\tilde E$ that arises from removing all such points. Thus, without loss of generality, we will always assume that the measure theoretic and topological boundary agree. 

The notation introduced in Subsections~\ref{ssec.blowup} and \ref{ssec.defsol} will be recurrent throughout the work. In particular, the definitions of $\Lambda_{\gamma,\theta}$ and $\Lambda^\delta$ from \eqref{eq.deflambda0}-\eqref{eq.deflambdad} as well as the definition of $e_w$ and the conditions on the constants $\theta$ and $\gamma$ (see \eqref{eq.eandtheta}). See also Figure~\ref{fig.lambda0}. 

On the other hand, when not stated otherwise, we add a superscript prime to an element or set in $\R^n$ to denote its projection to $\R^{n-1}$; and we proceed similarly with a double superscript prime projection to $\R^{n-2}$. Thus, if $x = (x_1,\dots,x_n) \in \R^n$, we can also denote $x = (x', x_n) \in \R^{n-1}\times\R$ or $x = (x'', x_{n-1}, x_n)\in \R^{n-2}\times\R\times\R$. Similarly, $B_1$ denotes the unit ball in $\R^n$, $B_1'$ is the unit ball in $\R^{n-1}$ and $B_1''$ in $\R^{n-2}$. We may sometimes write $B_1'\subset \R^n$, or $x'\in \R^n$ as an abuse of notation, meaning $B_1'\times\{0\}\subset \R^n$ and $(x', 0)\in \R^n$ respectively. 

\subsection{Preliminary results}

\begin{defi}
Let $E\subset \R^n$. We say that $E$ is a \emph{minimizer of the $\delta$-thin obstacle problem in} $B_1\subset\R^n$ if $\Phi(\Lambda^{\delta} ) \cap B_1\subset E$ and 
\eqref{eq.TOP} holds.
\end{defi}

We are also interested in the notion of super- and subsolutions to the minimal perimeter problem. Thus, the follow definition will also be useful. 

In general terms, we say that a set $E^+$ is a supersolution to the minimal perimeter problem when compact additive perturbations to $E^+$ in $B_1$ produce sets of larger perimeter. Similarly, $E^-$ is a subsolution to the minimal perimeter problem when compact subtractive perturbations to $E^-$ in $B_1$ increase the perimeter. 
\begin{defi}
\label{defi.sssolution}
Let $E^\pm \subset \R^n$. Then, $E^+$ is a \emph{supersolution} in $B$ if 
\[
P (F^+; B) \geq P (E^+; B),
\]
for any $F^+$ with $E^+\subset F^+$ and $\overline{F^+\setminus E^+}\Subset B$.

Analogously, $E^-$ is a \emph{subsolution} in $B$ if 
\[
P(F^-; B) \geq P(E^-; B),
\]
for any $F^-$ with $E^-\supset F^-$ and $\overline{E^-\setminus F^-}\Subset B$.
\end{defi}

Notice that, in particular, a set satisfying \eqref{eq.TOP} is a supersolution to the minimal perimeter problem.

\begin{prop}
Given $E_\circ\subset \R^n$ with $P(E_\circ;B_1)<\infty$, there exists $E$ satisfying \eqref{eq.TOP} with $E\setminus B_1= E_\circ\setminus B_1$. 
\end{prop}
\begin{proof}
The proof follows by classic methods in the calculus of variations. Lower semi\-continuity and compactness in $L^1$ of BV functions directly yield the result (see \cite[Thm 1.9, Thm 1.19]{Giu84}).
\end{proof}

\begin{prop}
\label{prop.minper}
Let $E\subset \R^n$ satisfying \eqref{eq.TOP}. Then, for any $B_r(x_\circ) \subset B_1$, $E$ is a \emph{supersolution} in $B_r(x_\circ)$. Moreover, if $B_r(x_\circ) \cap \Phi(\Lambda^{\delta} ) = \varnothing$, then $E$ is a set of minimal perimeter in $B_r(x_\circ)$.
\end{prop}
\begin{proof}
This just follows from the definitions of minimizer of the $\delta$-thin obstacle problem \eqref{eq.TOP} and supersolution.
\end{proof}

\begin{lem}
\label{lem.minvis}
If $E$ is a local minimizer of the perimeter around a point $x_\circ\in \de E$, then $\de E$ satisfies the mean curvature equation
\[
M(D^2 v , \nabla v) := (1+|\nabla v|^2)\Delta v - (\nabla v)^T D^2 v \nabla v = 0
\]
in the viscosity sense. That is, if we define for any smooth $\varphi: B_1' \to \R$, 
\[
S^{\pm}_\varphi := \{\pm x_n < \varphi(x')\},
\]
then, if $S^\pm_\varphi$ is included in either $E$ or $E^c$ in some ball $B_r(x_\circ)$ and $x_\circ \in \de S^\pm_\varphi$, we have that 
\begin{equation}
\label{eq.Mvis}
\pm M(D^2 \varphi , \nabla \varphi) \leq 0.
\end{equation}
Moreover, if $E$ is a supersolution to the minimal perimeter problem around $x_\circ\in \de E $, then if $S^\pm_\varphi$ is included in $E$ in some ball $B_r(x_\circ)$ and $x_\circ \in \de S^\pm_\varphi$ we have the same result, \eqref{eq.Mvis}.
\end{lem}
\begin{proof}
The proof is very standard, just using the definitions of minimal perimeter and supersolution and noticing that we can decrease the perimeter if the conclusion does not hold. See, for example, \cite{CC93}. 
\end{proof}


\begin{lem} \label{lem.asumpPhi}
Let $\Phi:\R^n \rightarrow \R^n$ be any $C^{1,1}$ diffeomorphism and let $x_\circ$ belong to $\partial \mathcal O = \Phi(\{ x_{n-1}= x_n=0\})$.  Assume that $[\Phi]_{C^{1,1}} \le M$ and $|D(\Phi^{-1})(x_\circ)| \le M$.
Then, for $\rho > 0$, there are new coordinates $\bar x = \psi_{x_\circ}(x)$
\[ \psi_{x_\circ}(x) :=  \rho^{-1}R_{x_\circ} (x -x_\circ), \quad \mbox{where $R_{x_\circ}$ is an orthogonal matrix}, \]
and a new $C^{1,1}$ diffeomorphism $\bar \Phi$, such that
\[
 \bar\Phi(\Lambda^{\bar\delta}) = \psi_{x_\circ}(\Phi(\Lambda^\delta))  \quad \mbox{for some }\bar \delta \in (0, C\delta)
\]
and 
\[
\bar \Phi( 0 ) = 0, \quad \bar \Phi( 0 ) = {\rm id}, \quad \mbox{and} \quad |D^2 \bar \Phi |\le  CM^3\rho,
\]
where $C$ depends only on $n$.
\end{lem}
\begin{proof}
Let us choose $R_{x_\circ}$ to be some orthogonal matrix to be chosen and define 
\[
A_{x_\circ} := R_{x_\circ} D\Phi (\Phi^{-1}(x_\circ)\big).
\]

Choose $R_{x_\circ}$ and $\bar \delta\in (0, C\delta)$ such that 
\[    A_{x_\circ}(\Lambda^\delta) = \Lambda^{\bar\delta} \]
as a consequence the set 
\[  \{x_{n-1} = 0, x_n\le 0\}  \quad \mbox{is invariant under the linear map }A_{x_\circ}. \]
Now define 
\[
\Phi^{x_\circ}:=  R_{x_\circ} \left( \Phi( \Phi^{-1}(x_\circ)+ A_{x_\circ}^{-1}x) -x_\circ\right) \quad  \mbox{and}\quad \bar \Phi:= \rho^{-1} \Phi^{x_\circ}(\rho x).
\]
Note that since  $\Phi^{-1}(x_\circ)\in \{x_{n-1}=x_n=0\}$  we have $\Phi^{-1}(x_\circ)+ A_{x_\circ}^{-1} \Lambda^{\bar \delta} = \Lambda^{\delta}$  and thus
\[\bar \Phi(\Lambda^{\bar \delta}) = \psi_{x_\circ}( \Phi( \Phi^{-1}(x_\circ)+ A_{x_\circ}^{-1} \Lambda^{\bar\delta}) ) = \psi_{x_\circ}(\Phi(\Lambda^\delta)) .\]

By construction, we have $ \bar \Phi(0) = 0$, $D \bar \Phi (0) = {\rm id}$, and $[ \bar \Phi]_{C^{1,1}} \le  CM^3\rho$.
\end{proof}

\section{Barriers and dichotomy}
\label{sec.3}
For this section let us start by defining the mean curvature operator $H$, on functions $\varphi:\R^{n-1}\to\R$ as
\begin{equation}
\label{eq.H}
H\varphi = {\rm div}\left(\frac{\nabla \varphi}{\sqrt{1+|\nabla\varphi|^2}}\right) = (1+|\nabla \varphi|^2)^{-\frac 3 2 }M(D^2\phi, \nabla \varphi).
\end{equation}
We start by introducing a supersolution that will be used as barrier. 

\begin{lem}[Supersolution]
\label{lem.supersolution}
Let $\beta \in \left(0, \frac{1}{10(n-2)}\right)$. Let
\[
\begin{split}
S^+_\beta & := \big\{x = (x'', x_{n-1}, x_n) \in B_1 \subset \R^{n-2}\times\R\times\R: \\
& ~~~~~~~~~~~~~~~~~~x_n \leq \varphi_\beta(x') := \beta\left(|x''|^2 - 2(n-2)x_{n-1}^2\right)\big\}
\end{split}
\]
Then, $S^+_\beta$ is a strict supersolution to the equation of minimal graphs in $B_1$, and 
\[
H\varphi_\beta \le -c\beta,\quad\textrm{in }~ B_1',
\]
for some positive constant $c$ depending only on $n$. 
\end{lem}
\begin{proof}
Let us check that, given $\varphi_\beta$, then 
\[
H\varphi_\beta \leq -c\beta.
\]

Let us rewrite the operator $H$,
\[
H\varphi_\beta (x') = \frac{1}{\sqrt{1+|\nabla \varphi_\beta|^2}}\left(\Delta \varphi_\beta - \frac{(\nabla \varphi_\beta)^T D^2 \varphi_\beta \nabla \varphi_\beta}{1+|\nabla \varphi_\beta|^2}\right) (x') = \sum_{i, j}U_{ij}(x') \de_{ij}\varphi_\beta(x'),
\]
where 
\[
U_{ij}(x') := \frac{1}{\sqrt{1+|\nabla \varphi_\beta|^2}} \left(\delta_{ij} - \frac{\de_i \varphi_\beta(x') \de_j \varphi_\beta(x')}{1+|\nabla \varphi_\beta|^2}\right).
\]

Let $S_\varphi(x')= \sqrt{1+|\nabla \varphi_\beta|^2}$. Note that, $U(x') = S_\varphi^{-1}(x') \left({\rm Id} - \bar \varphi_\beta \bar \varphi_\beta^T\right)$, where $\bar \varphi_\beta(x') = \nabla \varphi_\beta(x') / S_\varphi(x')$. The only eigenvalue of ${\rm Id} - \bar \varphi_\beta \bar \varphi_\beta^T$ different from 1 is $1-\|\bar \varphi_\beta\|^2$. Let $m_\varphi = \sup\{|\nabla \varphi_\beta|\}$, where the supremum is taken over the domain of definition of $\varphi_\beta$. Putting all together we have obtained that $U$ is uniformly elliptic, with ellipticity constants $\lambda_\varphi = (1+m_\varphi^2)^{-3/2}$ and 1.

Notice then that
\[
H\varphi_\beta(x') = \sum_{i, j}U_{ij}(x') \de_{ij}\varphi_\beta(x') \leq \beta\left(2(n-2)  - 4(n-2)\lambda_\varphi\right),\quad\textrm{in } B_1'. 
\]

On the other hand, from the fact that $|\nabla \varphi|\leq 4\beta(n-2)$ in $B_1'$,
\begin{equation}
\label{eq.lambdaphi}
\lambda_\varphi = (1+m^2_\varphi)^{-3/2} \geq (1+16\beta^2(n-2)^2)^{-3/2} .
\end{equation}

Putting all together, we get the desired result. 
\end{proof}

The following lemma shows that whenever the minimizer is not flat, then the contact set is full in the interior. The condition of flatness is used via the angle $\theta$ from the definition of the wedge $\Lambda_{{\gamma}, \theta}$: being flat means that $\theta$ is small, when compared to $\eps$.

\begin{lem}
\label{lem.cneps}
There exists $\eps_\circ$ and $C_\circ$ depending only on $n$ such that the following statement holds: 

Let $E\subset\R^n$ satisfying \eqref{eq.TOP} be such that it is $\eps$-close to some $\Lambda_{\gamma, \theta}$ in $B_1$, for some $\eps\in (0, \eps_\circ)$, and \eqref{eq.Phihyp} holds. Suppose that $\theta \in \big[C_\circ\eps, \frac{\pi}{2}\big)$. Then
\[
E \subset \Phi(\Lambda_{{\gamma}, \theta - C_\circ\eps})\quad\mbox{ in } B_{1/2}.
\]
In particular, the contact set is full in $B_{1/2}$.
\end{lem}
\begin{proof}
Let us prove this result, for simplicity, in the case $\Phi \equiv {\rm id}$, and at the end of the proof we discuss how to modify it in order to account for small second order perturbations. 

We will slide an appropriate supersolution from above until we intersect with the surface $\de E$. 

Take $x_\circ \in B_{1/2}''\times\{0\}\times\{0\}$, and by making a translation let us assume $x_\circ$ is the origin. Let us also rotate the setting with respect to the last two coordinates so that the angle between $e_\gamma$ and $\boldsymbol{e}_n$ is $\angle(e_\gamma, \boldsymbol{e}_n) = \theta- \arctan(\tilde C\eps)$, for some constant $\tilde C$ depending only on $n$ to be chosen, such that $\theta > \arctan(\tilde C\eps)$. Let us denote $e_\gamma^r$, $\de E^r$, $\de \Lambda_{{\gamma}, \theta}^r$, and $(\Lambda^{\delta})^r$, the corresponding rotated versions. The following argument can be done with both configurations that fulfil this property, so let us assume without loss of generality that we are in a situation where 
\begin{equation}
\label{eq.lss1}
\{x_n = -\tilde C\eps x_{n-1}\} \cap \{x_{n-1}\geq 0\}\subset \de \Lambda_{{\gamma}, \theta}^r,\quad\mbox{ in } B_{1/2}.
\end{equation}
See Figure~\ref{fig.super} for a representation of this rotated situation, and the whole proof. 

\begin{figure*}[t]
\centering
\includegraphics[width = 16cm]{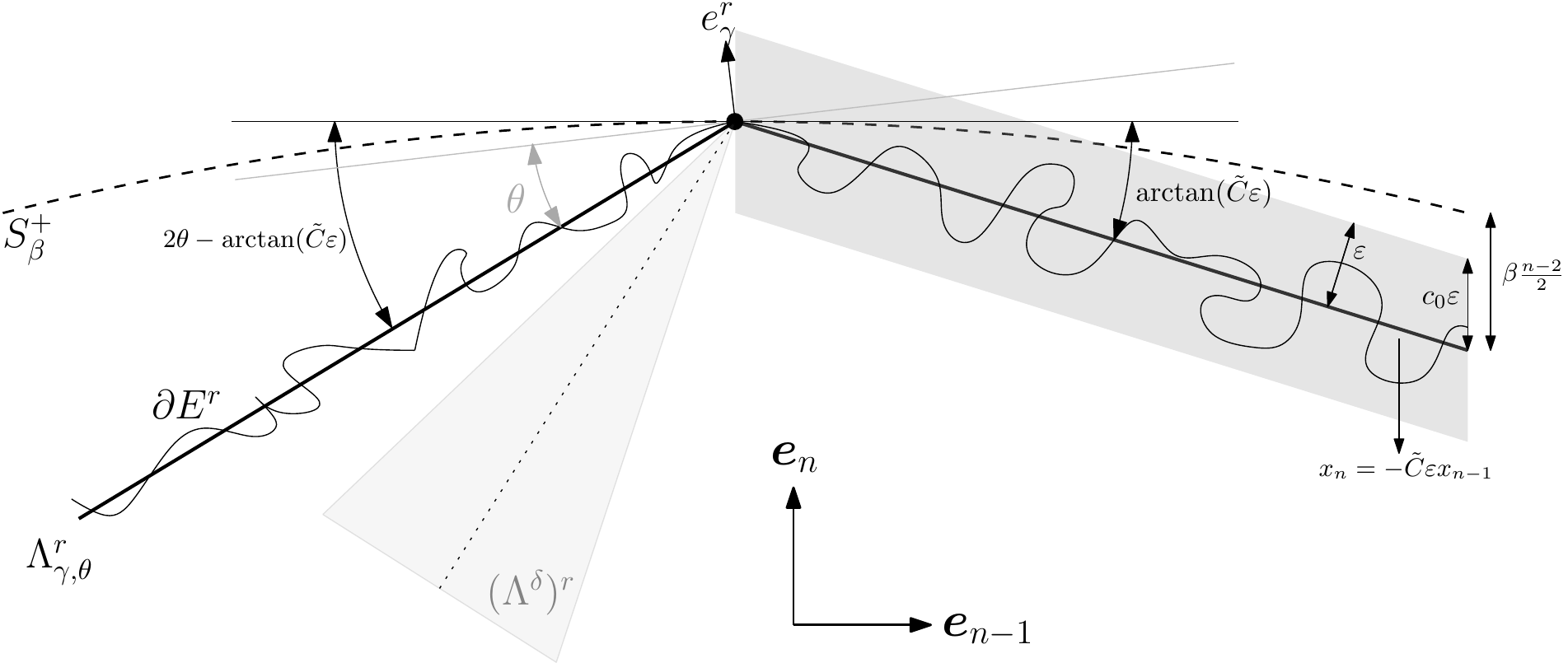}
\caption{Representation of the setting in Lemma~\ref{lem.cneps} after a rotation.}
\label{fig.super}
\end{figure*}

Take the supersolution $S^+_\beta$ from Lemma~\ref{lem.supersolution}. Slide $\de S^+_\beta$ from above until it touches the boundary of the minimizer of the $\delta$-thin obstacle problem, $\de E^r$. That is, define 
\[
S^{t}_\beta := \de S_\beta^+ + t\boldsymbol{e}_n,
\]
and consider 
\[
m_\beta := \inf \{t > 0: S^{t}_\beta\cap \de E^r\cap B_{1/2} \neq \varnothing\}.
\]
We recall that
\[
\de S_\beta^+ = \left\{x = (x'', x_{n-1}, x_n) \in B_1: x_n = \beta\left(|x''|^2 - 2(n-2)x_{n-1}^2\right)\right\}.
\]

If $m_\beta>0$ and $x^m = (x_1^m,\dots,x_n^m)\in B_{1/2}$ is such that $x_m \in S^{m_\beta}_\beta\cap \de E^r\cap B_{1/2}$, then $x^m$ cannot be an interior point to $S_\beta^{m_\beta}\cap B_{1/2}$. Indeed, since $S^{m_\beta}_\beta \cap B_{1/2}\cap \{x_{n-1} = 0\}\subset \{x_n \geq m_\beta > 0\}$ is strictly above zero, then thanks to Proposition~\ref{prop.minper} $\de E^r$ is a surface of minimal perimeter around $x_m$. On the other hand, $S^{m_\beta}_\beta$ is a supersolution, touching on an interior point with a surface of minimal perimeter locally, which is not possible. 

We will show that the boundary $\de B_{1/2} \cap S^{m_\beta}_\beta$ is always \emph{above} $\de E^r$ in the $\boldsymbol{e}_n$ direction. From \eqref{eq.lss1} and using that $\partial E^r \subset \Lambda_{{\gamma}, \theta}^r +B_\eps$, it is enough to show that there exists $\tilde C$ depending only on $n$ such that
\begin{equation}\label{eq.aaaa}
\beta\left(|x''|^2 - 2(n-1)x_{n-1}^2\right) \geq -\tilde C\eps x_{n-1} + c_0 \eps,\quad \mbox{for } x'= (x'', x_{n-1}) \in \de B_{1/2}',
\end{equation}
for some constant $c_0$ depending only on $n$ that accounts for the difference in distance between the Hausdorff distance and the distance in the $\boldsymbol{e}_n$-direction. For \eqref{eq.aaaa} to be satisfied, using $|x''|^2 = \frac 1 4 - (x_{n-1})^2$, we want
\[
-\beta(2n-1)x_{n-1}^2 + \tilde C\eps x_{n-1} \geq -\frac{\beta}{4} + c_0\eps,\quad\mbox{ for }x_{n-1}\in [0, 1/2].
\]
By taking $\beta = 4c_0 \eps$ and $\tilde C = 2c_0(2n-1)$ the previous condition holds, and notice that for $\eps $ small enough (depending only on $n$) $S^+_\beta$ is a supersolution as wanted. 

Thus, for $\beta = 4c_0 \eps$ and $\tilde C = 2c_0 (2n-1)$, we can slide $S^t_\beta$ until $t = 0$, where it touches $\de E^r$ at the origin (since it touches $(\Lambda^{\delta})^r$ there). Therefore, the origin is a contact point, and moreover, $\de E^r$ is contained in $S^+_\beta\cap \{x_{n-1} \geq 0\}$. In particular, since the origin was a translation of any point in $B_{1/2}''\times\{0\}\times\{0\}$, we have that in $B_{1/2}''\times\{0\}\times\{0\} \cap \{x_{n-1}\geq 0\}$, $\de E^r$ is contained in $\{x_n \leq 0\}$.

Rotating back, and putting $\arctan(\tilde C\eps) = C_\circ \eps$ for some $C_\circ$ depending only on $n$, we obtain the desired result from one side. Doing the same on the other side completes the proof. 

If $\Phi\not\equiv{\rm id}$, we can proceed similarly using that $|D^2\Phi|\le \eps^{1+\frac12}$. Indeed, if $E$ is $\eps$-close to $\Lambda_{{\gamma},\theta}$, then $\Phi^{-1}(E)$ is $2\eps$-close to $\Lambda_{{\gamma},\theta}$ for $\eps$ small enough depending only on $n$. Now we can repeat the previous argument with $\Phi^{-1}(E)$ instead of $E$. The only place where we used that $E$ satisfies \eqref{eq.TOP} is to check that we cannot touch at an interior point when sliding the supersolution (using the previous notation, to check that $m_\beta$ cannot be strictly positive). 

If we were touching at an interior point $x_m$ in this case, then $E$ would be a surface of minimal perimeter around $\Phi(x_m)$. Since we can choose $\beta = 4c_0\eps$ to avoid contact in the boundary, thanks to Lemma~\ref{lem.supersolution} the mean curvature of $\de S^{m_\beta}_\beta$ is below $-4c\eps$. Consequently, the mean curvature of $\Phi(\de S^{m_\beta}_\beta)$ is below $-4c\eps + c'\eps^{1+\frac12}$ and for $\eps$ small enough $\Phi(S_\beta^{m_\beta})$ is still a supersolution: there cannot be an interior tangential contact point. 
\end{proof}

Lemma \ref{lem.cneps} shows that if if $E$ is $\eps$-close to some wedge $\Lambda_{{\gamma},\theta}$ in $B_1$ with $\theta \ge C_\circ \eps$ then we have $E\subset\Phi(\Lambda_{\gamma, \theta- C_\circ \eps})$. As a counterpart, the following lemma shows that   $\Phi(\Lambda_{\gamma, \theta+ C_\circ \eps})\subset E$ --- even for $\theta<C_\circ \eps$.
\begin{lem}
\label{lem.subslem}
There exists $\eps_\circ$ and $C_\circ$ depending only on $n$ such that the following statement holds:

Let $E\subset\R^n$ satisfying \eqref{eq.TOP} be such that it is $\eps$-close to some $\Lambda_{{\gamma}, \theta}$ in $B_1$, for some $\eps\in (0, \eps_\circ)$ and $\theta \in \big[0, \frac{\pi}{2}-C_\circ\eps\big)$. Suppose that $\Phi$ satisfies \eqref{eq.Phihyp}. Then 
\[
\Phi(\Lambda_{{\gamma}, \theta + C_\circ\eps})\subset E \quad\mbox{ in } B_{1/2}.
\]
\end{lem}
\begin{proof}

The proof follows very similarly to the previous result, Lemma~\ref{lem.cneps}. Again, as before, we assume $\Phi \equiv {\rm id}$; and the proof can be adapted to the case $|D^2\Phi|\le \eps^{1+\frac12}$ following analogously to the proof of Lemma~\ref{lem.cneps}. 

We want to show that we can \emph{open} $\Lambda^{\delta} $ up to being at an angle proportional to $\eps$ from $\Lambda_{{\gamma}, \theta}$. Let us show it for $x_{n-1} \geq 0$.

The fact that $\Lambda^{\delta}  \subset E$ in $B_1$ allows us to establish a separation between $x_{n-1} \geq 0$ and $x_{n-1}\leq 0$.

Consider the surface $\de E \cap \{x_{n-1}\geq 0\}$. Let $\theta_1$ be the angle between $\de \Lambda_{{\gamma}, \theta}$ and $\de \Lambda^{\delta} $ in $\{x_{n-1}\geq 0\}$. If $\theta_1 \leq C_1 \eps$ for some $C_1$ depending only on $n$ we are already done, since $\Lambda^{\delta} $ is already a barrier; so that we can suppose that $\theta_1 \geq C_1\eps$ for some $C_1$ to be determined. We denote $\Gamma_{\gamma, \theta} = \de \Lambda_{{\gamma}, \theta} \cap \{x_{n-1} \geq 0\}$.

Now, as in Lemma~\ref{lem.cneps}, we rotate the setting in the last two coordinates, so that $\Gamma^r_{\gamma, \theta}\subset \{x_n \geq 0\}$ at an angle $\arctan(\tilde C\eps)$ from $\{x_n = 0\}$, for some constant $\tilde C$ to be chosen. See Figure~\ref{fig.sub} for a representation after the rotation. 

\begin{figure*}[t]
\centering
\includegraphics[width = 14cm]{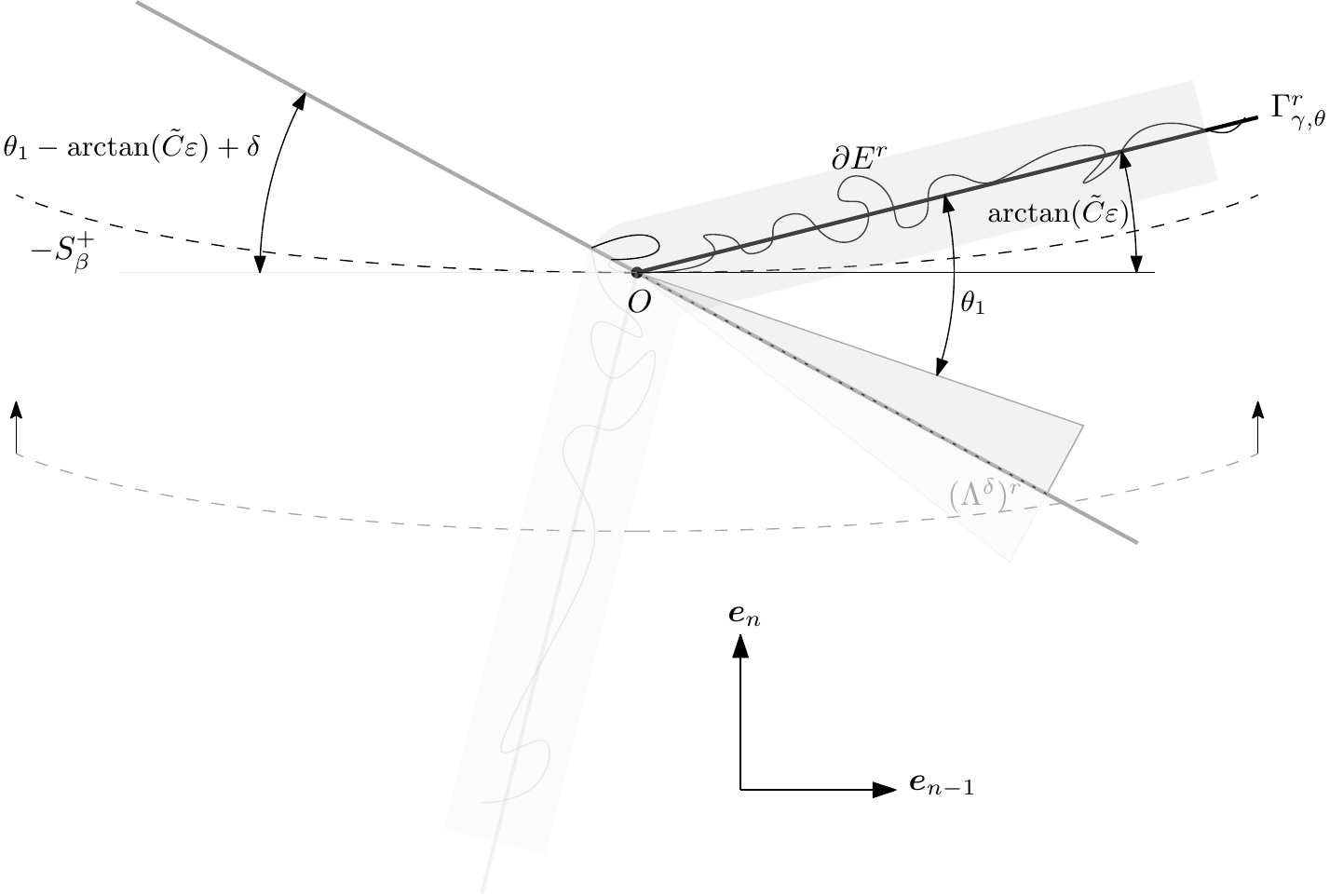}
\caption{Representation of the setting in Lemma~\ref{lem.subslem} after a rotation.}
\label{fig.sub}
\end{figure*}

Notice that $-S^+_\beta$ is a subsolution to the problem, where $S^+_\beta$ denotes the supersolution constructed in Lemma~\ref{lem.supersolution}. Now the situation is the same as in Lemma~\ref{lem.cneps} upside down. In the new coordinates after the rotation, since in $\{x_{n-1}> 0\}$ any point on $\de E^r$ is locally a supersolution, we will be able to slide up the subsolution up until the origin for the same constant $\tilde C$ as in Lemma~\ref{lem.cneps} as long as we are are not touching with it in the region $\{x_{n-1}\leq 0\}$ after the rotation. But this can be avoided choosing $C_1$ such that $C_1 \eps \geq 3 \arctan{\tilde C \eps}$ for $\eps$ small. 
\end{proof}

\section{Improvement of closeness in flat configuration}
\label{sec.4}
In this section we prove our main result, Theorem~\ref{thm.main}, in the flat configuration case in the case $\theta \in (0, C_\circ\eps)$. Namely, we show:

\begin{prop}
\label{prop.impclos}
For every $\alpha \in \left(0, \frac{1}{2}\right)$, there exist positive constants $\rho_\circ$ and $\eps_\circ$ depending only on $n$ and $\alpha$, such that the following statement holds:
 
Let $E\subset \R^n$ satisfying \eqref{eq.TOP}, with $0\in \de E$, be such that $E$ is $\eps$-close to $\Lambda_{{\gamma}, \theta}$ in $B_1$, for some $\theta \in (0, C_\circ\eps)$ and $\eps\in(0,\eps_\circ)$, and \eqref{eq.Phihyp} holds.
 
 Then, 
 \[
   E \mbox{ is } \rho_\circ^{1+\alpha} \eps\mbox{-close to } \Lambda_{\tilde \gamma, \tilde \theta}\mbox{ in } B_{\rho_\circ},
 \]
 for some new $\tilde \gamma'$ and $\tilde \theta$ as in \eqref{eq.eandtheta}. 
\end{prop}

The proof of this proposition follows by compactness, using the $C^{1,1/2}$ regularity of the solutions to the classical thin obstacle problem with the Laplacian, $\Delta$. 

The following proposition will be used to show compactness of vertical rescalings $\big\{ (x',x_n/\eps) \,:\, (x',x_n)\in \partial E \big\}$ near a contact point.

\begin{prop}
\label{prop.harnacktype}
 There exist $h_\circ$ and $\tau_\circ$ depending only on $n$ such that the following statement holds: 
 
Denote $Q_1 := B_1' \times(-1,1)$ Let $E\subset \R^n$ satisfying, for some $\boldsymbol v \in Q_1$,
 \begin{equation} \label{1.17'}
  P(E;Q_1)\le P(F;Q_1) \quad \forall F\  : \ E\setminus Q_1 = F\setminus Q_1 \mbox{ and } \big(\boldsymbol v +\Phi(\Lambda^{\delta})\big) \cap Q_1\subset F. 
\end{equation}
be such that for some $b\in(-1, 1)$ and some $h \in (0, h_\circ)$, \eqref{eq.Phihyp} holds for $\eps \in (0, h)$, 
 \[
  \{x_n \leq b-h \} \subset E \subset \{x_n \leq b + h\},\quad\mbox{ in } B_1'\times(-1, 1),
 \]
 and 
 \[
  \big(\boldsymbol v +\Phi(\Lambda_{0, h})\big) \subset E,\quad\mbox{ in } B_1'\times(-1, 1).
 \]
 Then, 
 \begin{itemize}
  \item either $\{x_n \leq b - h(1-\tau_\circ)\}\subset E$, in $B_{1/2}'\times (-1,1)$;
  \item or $E\subset \{x_n \leq b + h(1-\tau_\circ)\}$, in $B_{1/2}'\times (-1,1)$.
 \end{itemize}
\end{prop}

To prove Proposition~\ref{prop.harnacktype} we need the following half-Harnack for supersolutions; see \cite[Section 2]{Sav10b} or the proof of \cite[Thm 5.3]{Sav10}. 
 \begin{prop}[\cite{Sav10,Sav10b}]
 \label{prop.savinhh}
Let $E\subset \R^n$ be a supersolution to the minimal perimeter problem in $B_1$, and suppose $\de E \subset \{x_n \geq 0\}$. Then, for every $\eta_\circ > 0$, there exists some $\tau_\circ$ and $C$ depending only on $n$ and $\eta_\circ$ such that if $\tau < \tau_\circ$ and $\tau \boldsymbol{e}_n \in \de E$, then 
  \[
  \left|\Pi_{\boldsymbol{e}_n} \left(\de E \cap \{x_n \leq C \tau \}\cap (B_1'\times (-1,1)) \right) \right|_{\mathcal{H}^{n-1}} \geq (1-\eta_\circ) |B_1'|_{\mathcal{H}^{n-1}};
  \]
  where $\Pi_{\boldsymbol{e}_n} $ denotes the projection of a set onto $B_1'$ in the $\boldsymbol{e}_n$ direction. 
 \end{prop}

\begin{proof}[Proof of Proposition~\ref{prop.harnacktype}]
We separate the proof into two different scenarios. 

The first possibility is $b\leq \eps^{1+\frac14}$. In this case, since $\Phi(\Lambda_{0, h})\subset E$, it follows that 
\[
\left\{x_n\le -\frac{\tan h}{2} - C\eps^{1+\frac12}\right\}\subset E,\quad \mbox{ in } B_{1/2}'\times (-1, 1),
\]
for some $C$ depending only on $n$. For $h_\circ$ small enough depending only on $n$, since $\eps \le h\le h_\circ$ and $b \le \eps^{1+\frac14}$, 
\[
\left\{x_n\leq b-\frac{3}{4}h \right\}\subset \left\{x_n\leq -\frac{\tan h}{2}-C\eps^{1+\frac12} \right\}\subset E,\quad \mbox{ in } B_{1/2}'\times (-1, 1).
\]
This completes the case $b \le \eps^{1+\frac14}$.

The second case is $b > \eps^{1+\frac14}$, and is less straight-forward. By Savin's half Harnack, Proposition~\ref{prop.savinhh}, for every $\tau > 0$ small enough depending only on $n$, if there exists 
\begin{equation}
\label{eq.ch1}
z =(z', z_n)\in \de E,\mbox{ with } |z'|\leq\frac{1}{2}\mbox{ and } z_n\leq b-h+\tau h,
\end{equation}
then 
\begin{equation}
\label{eq.h1}
  \left|\Pi_{\boldsymbol{e}_n} \left(\de E \cap B_1\cap \left(B_{3/4}'\times(-1, 1)\right) \cap \{x_n \leq b-h+C_1 \tau h \}\right)\right|_{\mathcal{H}^{n-1}} \geq \frac{3}{4} |B_{3/4}'|_{\mathcal{H}^{n-1}},
\end{equation}
for some constant $C_1$ depending only on $n$. 

On the other hand, notice that since we are in the case $b > \eps^{1+\frac14}$, 
\[
\tilde E := E\cup \{x_n \leq b\},
\]
is a subsolution to the minimal perimeter problem in $B_1$ for $h$ small enough. This follows since $\Phi(\Lambda^{\delta} )\subset \{x_n \le \eps^{1+\frac14}\}$ for $\eps$ small enough, and $\de E$ is a surface of minimal perimeter whenever it does not touch $\Phi(\Lambda^{\delta} )$.

Take $\tilde E^c$, and apply again Proposition~\ref{prop.savinhh} to get that, for every $\tau > 0$ small enough depending only on $n$ (take $\tau < C_1^{-1}$), if there exists 
\begin{equation}
\label{eq.ch2}
z =(z', z_n)\in \de E,\mbox{ with } |z'|\leq\frac{1}{2}\mbox{ and } z_n\geq b+h-\tau h,
\end{equation}
then 
\begin{equation}
\label{eq.h2}
  \left|\Pi_{\boldsymbol{e}_n} \left(\de E \cap B_1\cap \left(B_{3/4}'\times(-1, 1)\right) \cap \{x_n \geq b+h-C_1 \tau h \}\right)\right|_{\mathcal{H}^{n-1}} \geq \frac{3}{4} |B_{3/4}'|_{\mathcal{H}^{n-1}}.
\end{equation}

Take $Q = B_{3/4}' \times (b-h, b+h)$
In particular, we must have that 
\[
P(E; Q) \geq \frac{3}{2}|B_{3/4}'|_{\mathcal{H}^{n-1}}.
\]

Notice, on the other hand, that we can take $h$ small enough so that the lateral perimeter of $Q$ is less than $\frac{1}{2}|B_{3/4}'|_{\mathcal{H}^{n-1}}$. This yields a contradiction, since including $Q$ to $E$ gives a competitor for the minimizer of \eqref{eq.TOP}; and therefore either \eqref{eq.ch1} or \eqref{eq.ch2} does not hold. This completes the proof. 
\end{proof}

We also need a similar improvement of oscillation \emph{far away from contact points}. In such case, we can use the following classical Harnack inequality for minimal surfaces. 
The proof of this proposition is an straightforward application of Proposition \ref{prop.savinhh}.

\begin{prop}[\cite{Sav10b}]
\label{prop.harnackminimal}
There exists $h_\circ$ and $\tau_\circ$ depending only on $n$ such that the following statement holds:
 
 Let $E\subset \R^n$ be a set of minimal perimeter in $B_1'\times(-1, 1)$, such that for some $b\in(-1, 1)$ and some $h \in (0, h_\circ)$
 \[
  \{x_n \leq b-h \} \subset E \subset \{x_n \leq b + h\},\quad\mbox{ in } B_1'\times(-1, 1).
 \]

 Then, 
 \begin{itemize}
  \item either $\{x_n \leq b - h(1-\tau_\circ)\}\subset E$, in $B_{1/2}'\times (-1,1)$;
  \item or $E\subset \{x_n \leq b + h(1-\tau_\circ)\}$, in $B_{1/2}'\times (-1,1)$.
 \end{itemize}
\end{prop}

Actually, to account for situations in which $\partial E$ may stick to $\partial \Phi(\Lambda_{\gamma,\theta})$, we need the following version of Proposition~\ref{prop.harnackminimal} for minimal surfaces with flat enough thin obstacles.

\begin{prop}
\label{prop.harnackthino}
There exists $h_\circ$ and $\tau_\circ$ depending only on $n$ such that the following statement holds: 

Assume that $\Phi$ satisfies \eqref{eq.Phihyp} with $\eps\in (0, h)$. 
 Let $E\subset \R^n$, satisfying 
 \[\Phi\big(\{x_n\le0\}\big)\cap Q_1\subset E\] where we denote $Q_r := B_r'\times(-1, 1)$, be a solution of 
 \[
 P(E; Q_1) \le P(F; Q_1)\quad \forall F \mbox{ such that } E\setminus Q_1= F\setminus Q_1,   \ \Phi\big(\{x_n\le0\}\big)\cap Q_1\subset F. 
 \]
Assume that for some $b\in(-1, 1)$ and some $h \in (0, h_\circ)$
 \[
  \{x_n \leq b-h \} \subset E \subset \{x_n \leq b + h\},\quad\mbox{ in } Q_1.
 \]

 Then, 
 \begin{itemize}
  \item either $\{x_n \leq b - h(1-\tau_\circ)\}\subset E$, in $Q_{1/2}$;
  \item or $E\subset \{x_n \leq b + h(1-\tau_\circ)\}$, in $Q_{1/2}$.
 \end{itemize}
\end{prop}
\begin{proof}
The proof is very similar to that of Proposition \ref{prop.harnackminimal} in \cite{Sav10b}. We sketch it.

Note that, by \eqref{eq.Phihyp} we have 
\[ 
\Phi\big(\{x_n=0\}\big) \subset \{|x_n|\le  \eps^{1+\frac 12} \} \quad \mbox{in }Q_1.
\]
Now, if $b\le 0$, since $\partial E$ is above $\Phi\big(\{x_n=0\}\big)$ in $Q_1$, we have 
$\{x_n\le -\eps^{1+\frac 12} \}\subset E$ in $Q_1$. Thus we obtain $\{x_n\le b-h(1-\tau_\circ)\}\subset E$ in $Q_1$ provided $\eps^{1+\frac 12} \le h(1-\tau_\circ)$, which is trivially satisfied if $\tau_\circ\le 1/2$ and $\eps<h<h_\circ \le 1/4$. In other words, the first alternative of the conclusion of the proposition holds whenever $b\le 0$.

Let us now consider the case $b\ge 0$. Note that we may suppose that the ``coincidence set'' $\partial E\cap \Phi\big(\{x_n=0\}\big)$ is nonempty in $Q_{3/4}$ since otherwise the result follows immediately from Proposition~\ref{prop.harnackminimal}, noting $\partial E$ would be a minimal boundary in $Q_{3/4}$.

Since $E$ is a supersolution in $Q_1$ satisfying $\{x_n\le -\eps^{1+\frac 12 }\}\subset E$ in $Q_1$ such that has some point $x_\circ =(x_\circ', x_{\circ,n}) \in \partial E\cap Q_{3/4}$ with $x_{\circ,n}\in (-\eps^{1+\frac 12}, \eps^{1+\frac 12})$, Proposition~\ref{prop.savinhh} (with a standard covering argument) yields 
\begin{equation}\label{eq.1}
 \left|\Pi_{\boldsymbol{e}_n} \left(\de E \cap \{x_n \leq C \eps^{1+\frac 12} \}\cap Q_{3/4}\right) \right|_{\mathcal{H}^{n-1}} \geq \frac{3}{4} |B_{3/4}'|_{\mathcal{H}^{n-1}}.
\end{equation}
At the same time, the set $\tilde E := E\cup \{x_n \le b+h/2\})$ is a subsolution in $Q_1$ since the contact set $\partial E \cap \partial \Phi\big(\{x_n=0\}\big)\cap Q_1$ is contained in $\{x_n\le \eps^{1+\frac 12}\}\subset \{x_n \le b+h/2\}$ (recall $b\ge 0$ and $\eps \le h$). Thus, either 
\begin{equation}\label{eq.aa}
E \subset  \tilde E \subset \big\{x_n \le b+h(1-\tau_\circ)\big\}\quad \mbox{in } Q_{3/4}
\end{equation}
or else, by Proposition~\ref{prop.savinhh} applied to $\tilde E^c$, we would have
\begin{equation}\label{eq.2}
 \left|\Pi_{\boldsymbol{e}_n} \left(\de \tilde E \cap \{x_n \ge b+h - C\tau_\circ h\}\cap Q_{3/4}\right) \right|_{\mathcal{H}^{n-1}} \geq \frac{3}{4} |B_{3/4}'|_{\mathcal{H}^{n-1}}.
\end{equation}
Now \eqref{eq.aa} clearly implies the conclusion of the proposition (first alternative). On the other hand, should \eqref{eq.2} hold then, by definition of $\tilde E$, \eqref{eq.2} would also hold with $\partial \tilde E$ replaced by $\partial E$ and thus we would find a contradiction with \eqref{eq.1} when taking $\tau_\circ$ small enough so that $b+h - C\tau_\circ h>  C \eps^{1+\frac 12}$ (recall $\eps<h<h_\circ$ small enough). Indeed, this contradiction argument --- which uses the minimality of $\partial E$ among boundaries of sets containing the obstacle --- is identical to the one given in the proof of Proposition~\ref{prop.harnacktype}.
\end{proof}

At this point, combining Proposition~\ref{prop.harnacktype} and Proposition~\ref{prop.harnackthino} we obtain the following lemma regarding the convergence of vertical rescalings to a H\"older continuous function.

\begin{lem}
\label{lem.Eeps}
Let $(E_k)_{k\in \N}$ be a sequence such that $E_k\subset \R^n$ satisfy \eqref{eq.TOP}, with $0\in \de E_k$, and with $\Phi_k$ such that \eqref{eq.Phihyp} holds for $\eps = \eps_k$. Suppose $E_k$ is $\eps_k$-close to $\Lambda_{\gamma_k, \theta_k}$ in $B_1$, with $\theta_k \in (0, \eps_k)$, and with $\eps_k \to 0$ as $k \to \infty$. Suppose also that $\Phi_k(\Lambda_{\gamma_k, \theta_k+\eps_k})\subset E_k$ in $B_1$. Let
\begin{equation}
\label{eq.Eeps}
E_k^{\eps_k} := \left\{\left(x', \frac{x_n}{2\eps_k}\right) : x = (x', x_n)\in E_k^r\cap B_1\right\},\quad\quad\mbox{ for all } k \in \N,
\end{equation}
where $E_k^r := R_{\gamma_k}(E_k)$, and $R_{\gamma_k}$ denotes the rotation of angle $\gamma_k$ in the last two coordinates bringing $e_{\gamma_k}$ to $\boldsymbol{e}_n$. 

Then, there exists $u \in C^{0, a}(\overline{B_{1/2}'})$ with $\|u\|_{C^{0, a}(\overline{B_{1/2}'})} \leq C$, for some $C$ depending only on $n$, such that
\begin{equation}
\label{eq.gammaclose}
\{x_n \leq u(x')-\eps_k^{\beta}\} \subset E^{\eps_k}_k \subset \{x_n \leq u(x')+\eps_k^{\beta}\},\quad\mbox{ in } B_{1/2}'\times(-1, 1),
\end{equation}
for some $a > 0$ and $\beta > 0$ depending only on $n$. 
\end{lem}
\begin{proof}
Let us define the cylinder $Q_r(x_\circ) = \left(B_r'(x_\circ')\times (-1, 1)\right)\cap B_1$ for any $x_\circ = (x_\circ',x_{\circ,n}) \in B_1$. Notice that, thanks to the hypotheses, for any $x_\circ \in \de E_k^r\cap B_{1/2}$,
\[
\de E_k^r \cap Q_{1/2}(x_\circ^r) \subset \{x \in B_1: |x_n - x_{\circ, n}| \leq 2\eps_k\},
\]
where $x_\circ^r$ denotes the rotated version of $r$. 
That is, introducing a notation, we have
\[
\underset{Q_{2^{-1}}(x_\circ^r)}{{\rm osc}_n} \de E_k^r \leq 2\eps_k;
\]
the oscillation in the $\boldsymbol{e}_n$ direction of $\de E_k^r$ in the cylinder $Q_{2^{-1}}(x_\circ^r)$ is less than $2\eps_k$. We would like to use that if $\eps_k$ is small enough, then either Proposition~\ref{prop.harnacktype} or Proposition~\ref{prop.harnackthino} improves the oscillation in the half cylinder, and proceed iteratively. In order to do that, we separate between four cases.
\\[0.2cm]
{\it Case 1: $x_\circ = 0$.} The first case we consider is $x_\circ = 0\in \de E_k$. By assumption, $\Phi_k(\Lambda_{\gamma_k, \theta_k+\eps_k})\subset E_k$ in $B_1$, and we have that 
\[
\underset{Q_{2^{-1}}(x_\circ^r)}{{\rm osc}_n} \de E_k^r \leq 2\eps_k.
\]
If we denote as $h_\circ$ and $\tau_\circ$ the variables coming from Proposition~\ref{prop.harnacktype}; we have that if 
\begin{equation}
\label{eq.hyp00}
4\eps_k\leq h_\circ,
\end{equation}
then 
\[
\underset{Q_{2^{-2}}(x_\circ^r)}{{\rm osc}_n} \de E_k^r \leq 2\eps_k(1-\tau_\circ).
\]
We are using here Proposition~\ref{prop.harnacktype} with $h = \eps_k$. Condition \eqref{eq.hyp00} is to ensure that $\theta_k +\eps_k \le h_\circ$ \footnote{Notice that here we want to ensure that $\Phi(\Lambda_{0,h})\subset E_k^r$ in order to apply Proposition~\ref{prop.harnacktype}. We actually have that $R_{\gamma_k} \Phi_k(\Lambda_{\gamma_k, \theta_k+\eps_k})\subset E_k^r$, but this is enough to use it as a barrier from below in the proof of Proposition~\ref{prop.harnacktype}. }. If we rescale by a factor 2, we have 
\[
\underset{Q_{2^{-1}}(x_\circ^r)}{{\rm osc}_n} 2\de E_k^r \leq 4\eps_k(1-\tau_\circ),
\]
so that, if we want to repeat the argument, hypothesis \eqref{eq.hyp00} becomes 
\[
8\eps_k (1-\tau_\circ)\le h_\circ.
\]
If we want to continue one next iteration, we can take $h = 2\eps_k(1-\tau_\circ)$. Notice that, after the rescaling, the transformation $\Phi$ associated to $2\de E_k$, is $\tilde \Phi_k(x) = 2\Phi_k(x/2)$, so that $|D^2\tilde \Phi_k|\le 2^{-1}\eps_k^{1+\frac12}$, and the hypotheses of Proposition~\ref{prop.harnacktype} are still fulfilled, with a better constant. 

Rescaling and repeating this procedure iteratively, we have that as long as 
\begin{equation}
\label{eq.condrec}
2^m (1-\tau_\circ)^{m-2} \eps_k \leq h_\circ,
\end{equation}
then 
\begin{equation}
\label{eq.resrec}
\underset{Q_{2^{-m}}(x_\circ^r)}{{\rm osc}_n} \de E_k^r \leq 2\eps_k(1-\tau_\circ)^{m-1}.
\end{equation}
\\
{\it Case 2: $x_\circ \in \de E_k\cap \de \mathcal{O}_k\cap B_{1/2}$.} The second case is when $x_\circ$ belongs to the contact set of the thin obstacle, $x_\circ \in \de E_k\cap \de \mathcal{O}_k$, where $\de\mathcal{O}_k := \Phi(\{x_{n-1} = x_n = 0\})$. After a translation and a rotation, up to redefining $\Phi$ if necessary, we can put ourselves in Case 1 (see Lemma~\ref{lem.asumpPhi} with $\rho = 1$), so that
\begin{equation}
\label{eq.condrec2}
2^m (1-\tau_\circ)^{m-2} \eps_k \leq h_\circ\quad \Rightarrow  \quad \underset{Q_{2^{-m}}(x_\circ^r)}{{\rm osc}_n} \de E_k^r \leq 2\eps_k(1-\tau_\circ)^{m-1}.
\end{equation}
We must point out here that, a priori, the oscillation might be in a direction different from $\boldsymbol{e}_n$ due to the rotation coming from Lemma~\ref{lem.asumpPhi}. However, since the rotation tends to the identity as $\eps_k\downarrow 0$, we may also assume that for $\eps_k$ small enough, the previous also holds. 
\\[0.2cm]
{\it Case 3: ${\rm dist}(x_\circ, \de E_k\cap \de \mathcal{O}_k) \ge \frac18$.} Follows exactly as the two previous cases, using Proposition~\ref{prop.harnackthino} instead of Proposition~\ref{prop.harnacktype}, yielding again \eqref{eq.condrec2}.
\\[0.2cm]
{\it Case 4: $2^{-p-1}\le {\rm dist}(x_\circ, \de E_k\cap \de \mathcal{O}_k) \le 2^{-p}$ for $p \ge 3$.} This is a combination of Case 2 and Case 3. We apply Case 2 and rescale, until we can apply Case 3, so that \eqref{eq.condrec2} holds again.

That is, \eqref{eq.condrec2} holds for all $x_\circ \in \de E_k \cap B_{1/2}$. Let $m_k$ denote the largest $m$ we can take for every $\eps_k$ such that \eqref{eq.condrec} holds. Clearly, $m_k\to \infty$ as $k \to \infty$, since $\eps_k \to 0$. If we consider the rescaled sets in the $\boldsymbol{e}_n$ direction, $E_k^{\eps_k}$, we have that for every $m \leq m_k$, 
\begin{equation}
\label{eq.resrec2}
\underset{Q_{2^{-m}}(x_\circ)}{{\rm osc}_n} \de E_k^{\eps_k} \leq 2(1-\tau_\circ)^{m-1}.
\end{equation}

In particular, there exists a H\"older modulus of continuity as $\eps_k \to 0$ controlling the boundaries $\de E_k^{\eps_k}$. By Arzel\`a-Ascoli, up to subsequences, $\de E_k^{\eps_k}$ converges in the Hausdorff distance to the graph of some H\"older continuous function, $u$.
\end{proof}

\begin{lem}
\label{lem.ctop}
The function $u \in C^{0, a}(\overline{B_{1/2}'})$ from the Lemma~\ref{lem.Eeps} is a viscosity solution to the classical thin obstacle problem with $u(0) = 0$. That is, $u$ fulfils
\begin{equation}
  \label{eq.thinobst}
  \left\{ \begin{array}{rcll}
  \Delta u &=&0& \textrm{ in } B_{1/2}' \setminus \left(\{x_{n-1} = 0 \}\cap \{u = 0\}\right)\\
  \Delta u &\leq&0& \textrm{ on } \{x_{n-1} = 0 \}\cap \{u = 0\}\\
  u&\geq&0& \textrm{ on } \{x_{n-1} = 0 \},
  \end{array}\right.
\end{equation}
in the viscosity sense. In particular, 
\begin{equation}
\label{eq.regtopclass}
\|u\|_{C^{1,1/2}\left(\overline{B_{1/4}'\cap \{x_{n-1}\geq 0\}}\right)}+\|u\|_{C^{1,1/2}\left(\overline{B_{1/4}'\cap \{x_{n-1}\leq 0\}}\right)}\leq C,
\end{equation}
for some constant $C$ depending only on $n$. That is, $u$ is $C^{1,1/2}$ up to $\{x_{n-1} = 0\}$ in either side.
\end{lem}
\begin{proof}
The proof follows along the lines of \cite{Sav10}.

Since $\de E_k^{\eps_k}$ converges uniformly to the graph of $u$, and $\de E_k^{\eps_k} \cap \{x_{n-1} = 0\} \subset \{x_n \geq -C\eps_k\}$, we clearly have that $u \geq 0$ on $\{x_{n-1} = 0\}$. This follows since $\Phi(\Lambda_{\gamma_k, \theta_k+\eps_k})\subset E_k$. Similarly, $u(0) = 0$. 

Now take any point $x_\circ' \in B_{1/2}'$. Consider $P(x')$ a quadratic polynomial in $B_{1/2'}$, with graph touching the graph of $u$ from below at $(x_\circ', u(x_\circ'))$. Since $\de E_k^{\eps_k}$ is converging uniformly to the graph of $u$, $P(x')-c_k$ touches from below $\de E_k^{\eps_k}$ at a point $y_k$ such that $y_k \to (x_\circ', u(x_\circ'))$ as $k \to \infty$. Rescaling back, $\eps_k P(x') - \tilde c_k$ touches from below $\de E_k^r$ at $\tilde y_k$ such that $\tilde y_k' \to x_\circ'$ for some sequence $\tilde c_k$ bounded. Since $\de E_k^r$ is a supersolution being touched from below, by Lemma~\ref{lem.minvis} we have 
\[
M(\eps_k D^2 P, \eps_k \nabla P) = \eps_k \Delta P + \eps_k^3 \left(\Delta P |\nabla P|^2 - \eps_k (\nabla P)^T D^2 P\,\nabla P \right)\leq 0
\]
at $\tilde y_k'$. By letting $\eps_k \to 0$ we reach 
\[
\Delta P (x_\circ') \leq 0,
\]
so that $u$ solves $\Delta u \leq 0$ in the viscosity sense. 

On the other hand, suppose $x_\circ' \in B_{1/2}'\setminus\left( \{x_{n-1} = 0 \}\cap \{u = 0\}\right)$. Let $P(x')$ be a quadratic polynomial in $B_{1/2'}$, with graph touching the graph of $u$ from above at $(x_\circ', u(x_\circ'))$. Now, $P(x')+c_k$ touches from above $\de E_k^{\eps_k}$ at a point $y_k$ such that $y_k \to (x_\circ', u(x_\circ'))$ as $k \to \infty$. That is, $\eps_k P(x') + \tilde c_k$ touches from above $\de E_k^r$ at $\tilde y_k$ such that $\tilde y_k' \to x_\circ'$ for some sequence $\tilde c_k$ bounded. If $k$ large enough, $\tilde y'_k \in B_{1/2}'\setminus\left( \{x_{n-1} = 0 \}\cap \{u = 0\}\right)$. Therefore, either $\de E_k^r$ is a surface of minimal perimeter around $\tilde y_k$, or $\de E_k^r$ is touching $\Phi_k(\Lambda^{\delta} )$ at $\tilde y_k$. In the first case, we are already done proceeding as before, we get $M(\eps_k D^2 P, \eps_k\nabla P) \ge 0$.

Suppose then, that $\de E_k^r$ is touching $\Phi_k(\Lambda^{\delta} )$ at $\tilde y_k$. For this to happen, one must have that $\Phi_k(\Lambda^{\delta} )$ is a supersolution to the minimal perimeter problem around $\tilde y_k$, otherwise there could not be a contact point with a supersolution. However, notice that it is a supersolution with mean curvature around $\tilde y_k$ bounded from below by $-C\eps_k^{1+\frac12}$. Therefore, $M(\eps_k D^2 P, \eps_k \nabla P) \ge -C\eps_k^{1+\frac12}$ at $\tilde y_k$, and letting $k\to \infty$ we get $\Delta P(x_\circ')\ge 0$. Thus, \eqref{eq.thinobst} holds in the viscosity sense. 

Finally, the regularity of solution to the classical thin obstacle problem, \eqref{eq.regtopclass}, was first shown by Caffarelli in \cite{Caf79}; and the optimal $C^{1, 1/2}$ regularity here presented was obtained by Athanaopoulos and Caffarelli in \cite{AC04}. 
\end{proof}

We can now present the proof regarding the improvement of closeness to sets of the form $\Lambda_{{\gamma}, \theta}$, Proposition~\ref{prop.impclos}.

\begin{proof}[Proof of Proposition~\ref{prop.impclos}]
Let us argue by contradiction, and suppose that the statement does not hold. Then, there exists some $\alpha_\star\in \left(0,\frac{1}{2}\right)$ and a sequence $E_k\subset \R^n$ satisfying \eqref{eq.TOP}, such that $0\in\de E_k$, $E_k$ are $\eps_k$-close to some $\Lambda_{\gamma_k, \theta_k}$ for $\theta_k\in(0, C_\circ\eps_k)$, \eqref{eq.Phihyp} holds for $\eps = \eps_k$ (and the transformation $\Phi_k$), for some positive sequence $\eps_k \to 0$ as $k\to \infty$, but such that the conclusion does not hold for any $\rho_\circ, \eps_\circ >0$. 

By Lemma~\ref{lem.subslem} we have that 
\[
\Phi_k(\Lambda_{\gamma_k , \theta_k + C_\circ \eps_k}) \subset E_k,\quad\quad\mbox{ in } B_{1/2}.
\]
By rescaling and renaming the $\eps_k$ sequence if necessary, we can assume that $\theta_k\in(0, \eps_k)$ and $\Phi_k(\Lambda_{\gamma_k , \theta_k + \eps_k}) \subset E_k$ in $B_1$, so that we are in the same situation as in Lemma~\ref{lem.Eeps}. In particular, due to Lemma~\ref{lem.Eeps}, the sequence $\de E_k^{\eps_k}$ approaches (in Hausdorff distance) a function $u$ in $B_{1/2}'\times(-1, 1)$, which by Lemma~\ref{lem.ctop} is a solution to a classical thin obstacle problem. Thanks to the regularity of $u$, and the fact that $u(0) = 0$ and $\nabla_{x''} u (0) = 0$, we have that 
\[
\left|u(x') - \de^+_{n-1} u (0)(x'_{n-1})_+ - \de^-_{n-1} u (0)(x'_{n-1})_-\right| \leq C \rho^{3/2},\quad \mbox{in } B_{2\rho}',
\]
for any $\rho> 0 $ and for some constant $C$ depending only on $n$. Here, we have denoted $a_+ = \max\{a, 0\}$, $a_- = \min\{ a, 0\}$, and 
\[
\de^\pm_{n-1} u (0) := \lim_{\eta\downarrow 0} \frac{\de u}{\de x_{n-1}'} (0,\dots,0,\pm \eta),
\]
i.e., the limit of the derivative in the $\boldsymbol{e}_{n-1}$ direction coming from $\{x_{n-1} > 0\}$ or $\{x_{n-1} < 0\}$ (which exist by the regularity up to the contact set). Notice, moreover, that since $\Delta u\leq 0$ around 0, we must have $\de^-_{n-1} u (0) \geq \de^+_{n-1} u (0)$. In particular, thanks to the closeness of $\de E_k^{\eps_k}$ to the graph of $u$, we have that 
\[
\de E_k^{\eps_k} \cap \left(B_{3\rho/2}'\times(-1, 1)\right) \subset \left\{\left| x_n - \de^+_{n-1} u (0)(x'_{n-1})_+ - \de^-_{n-1} u (0)(x'_{n-1})_-\right| \leq C\rho^{1/2}\right\},
\]
which, after rescaling implies that $\de E_k^r$ is at distance at most $C\eps_k \rho^{3/2}$ from some $\Lambda_{\tilde{\gamma}, \tilde{\theta}}$ in $B_\rho$, given by the graph of $\eps_k \de^+_{n-1} u (0)(x'_{n-1})_+ + \eps_k \de^-_{n-1} u (0)(x'_{n-1})_-$. Now, simply take $\rho$ small enough depending only on $n$ and $\alpha_\star$ such that $C\rho^{3/2} \leq \rho^{1+\alpha_\star}$, and we reach a contradiction (notice that such $\rho$ exists because $\alpha_\star < \frac{1}{2}$).
\end{proof}

\section{Improvement of closeness in non-flat configuration}
\label{sec.5}
In this section we study the complementary case to the one in the previous section: the case where $E$ is $\eps$-close to a \emph{non-flat} ($\theta\gtrsim \eps$) wedge $\Lambda_{{\gamma}, \theta}$. Under this condition, thanks to Lemma~\ref{lem.cneps}, there exists a full contact set, so that the study of the regularity becomes a known matter. 

We state and prove now the lemma that will allow us to conclude the proof of Theorem~\ref{thm.main}. 

\begin{lem}
\label{lem.gamma}
There exists $\eps_\circ$ depending only on $n$ such that the following statement holds: 

Let $E \subset \R^n$ satisfying \eqref{eq.TOP} with $0\in \de E$ be such that for some $\Lambda_{{\gamma}, \theta}$, and $\eps \in (0, \eps_\circ)$,
\begin{equation}
\label{eq.gamma0}
\Phi(\Lambda_{{\gamma}, \theta+\eps})\subset E\subset \Phi(\Lambda_{{\gamma}, \theta-\eps}),\quad\mbox{ in } B_1,
\end{equation}
where $\Phi$ satisfies \eqref{eq.Phihyp}.

Then,
\begin{equation}
\label{eq.gamma1}
\de E \cap \overline{B_{1/2}} = \overline{\Gamma_+}\cup \overline{\Gamma_-},
\end{equation}
where 
\begin{equation}
\label{eq.gamma2}
\Gamma_{\pm} = \de E \cap B_{1/2} \cap \Phi(\{\pm x_{n-1} > 0\}),
\end{equation}
and 
\begin{equation}
\label{eq.gamma3}
\overline{\Gamma_\pm} \cap \Phi(\{x_{n-1} = 0\}) \cap \overline{B_{1/2}}\subset \Phi(\{x_{n-1} = x_n = 0\}).
\end{equation}

Moreover, for each $\beta\in (0,1)$, $\Gamma_+$ and $\Gamma_-$ are $C^{1, \beta}$ graphs up to the boundary in the $e_{\gamma+\theta}$ and $e_{\gamma-\theta}$ directions respectively,  with $C^{1,\beta}$-norms bounded by $C\eps$, where $C$ depends only $n$ and $\beta$.
\end{lem}

\begin{rem}
\label{rem.cinf}
A a direct consequence of the $C^{1,\beta}$ estimates from Lemma \ref{lem.gamma} there exists $\Lambda_{ \gamma_\star, \theta_\star}$ as in \eqref{eq.eandtheta} such that for any $\bar \alpha \in (0, 1/2)$,
\[
E \mbox{ is } C\eps r^{1+\bar\alpha}\mbox{-close to }\Lambda_{\gamma_\star, \theta_\star} \mbox{ in } B_r,\quad\mbox{ for all } r \in (0, 1/2),
\]
for some constant $C$ depending only on $n$. Moreover, 
\[
|\bar \gamma - \gamma|+|\bar \theta - \theta|\leq C \eps,
\]
for some constant $C$ depending only on $n$. This will be useful later on in the paper. In fact, we could clearly take $\bar \alpha\in (0,1)$ but we will only need $\bar \alpha<1/2$ later on (see Proposition \ref{prop.wow}).
\end{rem}

In order to prove Lemma~\ref{lem.gamma} we need a version for thick smooth obstacles of the following standard result on regularity of flat minimizers of the perimeter.

\begin{thm}[{\cite[Chapter 8]{Giu84}}]
\label{thm.giuch8}
There exists $\eta_\circ$ small depending only on $n$ such that the following statement holds:

Let $E\subset\R^n$ be a minimizer of the perimeter in $B_1$ such that 
\[
\{x_n \leq -\eta\} \subset E \subset \{x_n \leq \eta\},\quad\mbox{in } B_1,
\]
for some $\eta \in (0, \eta_\circ)$.

Then, there exists a map $\varphi : B_{1/2}' \to \R$ such that
\[
\de E = \{x = (x', x_n)\subset \R^n : x_n = \varphi(x')\}\quad\mbox{ in } B_{1/2}'\times\left(-1/2,1/2\right),
\]
where $\|\varphi\|_{C^k(B_{1/2}')} \leq C(n, k)\,\eta$, for some constant $C$ depending only on $n$ and $k$.
\end{thm}

Let us comment on the standard proof of the previous theorem.
\begin{rem}
\label{rem.impflat}
Theorem \ref{thm.giuch8} is usually shown in two steps. 
First, one iterates \eqref{eq.improofflatness} obtain
\[
|\nu(x)-\nu(y)|\leq C\eta|x-y|^\alpha,
\]
for $\alpha > 0$, and where $\nu(x)$ for $x \in \de E$ denotes the unit normal vector to $\de E$ pointing outwards $E$.
This $C^\alpha$ estimate for the normal $\nu$ is a consequence of the improvement of flatness property \eqref{eq.improofflatness}. 

Second, one improves this $C^{1,\alpha}$ estimate to obtain the $C^k$ regularity using interior Schauder estimates for graphs.

Comparing normal vectors is like comparing the corresponding tangent hyperplanes (or half-spaces). A similar approach is what inspired part of this work, where we compare sets of the form $\Lambda_{{\gamma}, \theta}$ instead of half-spaces to get the regularity. 
\end{rem}

The version of the previous result we will need is the following
\begin{thm}
\label{thm.flat_obst_p}
There exists $\eta_\circ$ small depending only on $n$ such that the following statement holds: 

Assume $\eta\in(0,\eta_\circ)$ and that $\Phi$ satisfies \eqref{eq.Phihyp} with $\eps\in(0,\eta)$. 
 Let $E\subset \R^n$, satisfying 
 \[\Phi\big(\{x_n\le0\}\big)\cap  B_1\subset E,\]
 \[
 P(E; B_1) \le P(F; B_1)\quad \forall F \mbox{ such that } E\setminus B_1= F\setminus B_1,   \ \Phi\big(\{x_n\le0\}\big)\cap B_1\subset F. 
 \]
Assume that for some $b\in(-1/2, 1/2)$
 \[
  \{x_n \leq b-\eta \} \subset E \subset \{x_n \leq b + \eta\},\quad\mbox{ in } B_1.
 \]

Then, there exists a map $\varphi : B_{1/2}' \to \R$ such that
\begin{equation}\label{eq.graph!!!}
\de E = \{x = (x', x_n)\subset \R^n : x_n = \varphi(x')\}\quad\mbox{ in } B_{1/2}'\times\left(b-1/4,b+1/4\right),
\end{equation}
where $\|\varphi\|_{C^{1,1}(B_{1/2}')} \leq C\eta$, for some constant $C$ depending only on $n$.
\end{thm}

The proof of Theorem \ref{thm.flat_obst_p} is based on two steps as the proof of Theorem \ref{thm.giuch8} (see Remark \ref{rem.impflat}). First, 
we prove that $\partial E$ is a $C^{1,\alpha}$ graph or, more precisely, \eqref{eq.graph!!!} with $\|\varphi\|_{C^{1,\alpha}(B_{1/2}')} \leq C\eta$. 
This can be done exactly by compactness of vertical rescaling, following the exact same strategy of Savin \cite{Sav10, Sav10b}.

Second, we can apply a theorem of Br\'ezis and Kinderleher \cite{BK} to improve from this $C^{1,\alpha}$ estimate to the optimal $C^{1,1}$ estimate. By completeness we sketch the proof here.

\begin{proof}[Proof of Theorem \ref{thm.flat_obst_p}] We do the argument in two steps.
\\[0.2cm]
{\bf Step 1.} Fix some $\alpha\in (0,1)$, say $\alpha:=1/4$. Then, we claim that if $\eta_\circ$ is small enough then \eqref{eq.graph!!!} holds with $\|\varphi\|_{C^{1,\alpha}(B_{1/2}')} \leq C\eta$, where $C$ depends only on $n$.
Indeed, exactly as in the proof of Proposition~\ref{prop.impclos}, we establish by compactness the following improvement of flatness property, around $x_\circ\in  B_{3/4}\cap \partial E$,
\begin{equation}\label{eq.impr111}
\partial E \subset \big\{ |e\cdot (x-x_\circ) |\le \eta \big\}  \mbox{ in }B_r(x_\circ) \quad \Rightarrow \quad \partial E \subset \big\{ |\tilde e\cdot (x-x_\circ)| \le \rho_\circ^{1+\alpha} \eta \big\}  \mbox{ in }B_{\rho_\circ r}(x_\circ). 
\end{equation}
for some $\rho_\circ\in (0,1)$ depending only on $n$. 
The proof of \eqref{eq.impr111} is analogous to the Proof of Proposition~\ref{prop.impclos}. It is enough to do the case $r=1$. To do it, we consider the vertical rescalings defined similarly as in \eqref{eq.Eeps} in Lemma \ref{lem.Eeps}.
These vertical rescalings of $\partial E$ are compact by Proposition \ref{prop.harnackthino} (similarly as in Lemma \ref{lem.Eeps}) and converge ``uniformly'' to a function $u\in C^a(B'_{1/2})$
which is harmonic. Indeed, the condition $|D^2\Phi|\le \eta^{1+\frac 1 2}$ implies that the thick obstacle will be zero in the limit if we apply the vertical rescaling $(x', x_n) \mapsto (x', x_n/\eta)$ and let $\eta\downarrow 0$.
Using the $C^{1,1}$ regularity of harmonic functions we establish \eqref{eq.impr111}.

With a standard iteration of \eqref{eq.impr111} we establish that \eqref{eq.graph!!!} holds with 
\[\|\varphi\|_{C^{1,\alpha}(B_{1/2}')} \leq C\eta \quad \mbox{ ($\alpha=1/4$)},
\]
as we wanted to show.
\\[0.2cm]
{\bf Step 2.} We improve the previous $C^{1,1/4}$ estimate to the optimal estimate $\|\varphi\|_{C^{1,1}(B_{1/2}')} \leq C\eta$. This is a straightforward application of the results of 
Br\'ezis and Kinderleher \cite{BK} of optimal $C^{1,1}$ regularity for obstacle problems with uniformly elliptic nonlinear operators. Indeed, once we have proved that $\partial E$ is a graph and 
with bounded gradient, then it follows that the mean curvature operator $H$ is uniformly elliptic and thus \cite[Theorem 1]{BK} provides exactly the desired $C^{1,1}$ estimate.
\end{proof}

We can now prove Lemma~\ref{lem.gamma}. 

\begin{proof}[Proof of Lemma~\ref{lem.gamma}]
We divide the proof into two steps. In the first step we show that $\Gamma_\pm$ are a graphs, and in the second step we show their regularity. 
\\[0.2cm]
{\bf Step 1: $\Gamma_\pm$ are graphs in an appropriate direction.} The proof of the fact that $\Gamma_\pm$ are graphs is almost immediate, just noticing that \eqref{eq.gamma0} allows us to apply Theorem~\ref{thm.flat_obst_p} at every scale.

Let us consider first the case $\Phi \equiv {\rm id}$, and let us rotate the setting with respect to the last two coordinates, in such a way that the normal vector to $\Lambda_{{\gamma}, \theta}$ for $\{x_{n-1} > 0\}$, $e_{\gamma+\theta}$, now becomes $\boldsymbol{e}_n$ (that is, rotate an angle $\gamma+\theta$). Let us denote as the corresponding rotated versions with superindex $r$, e.g. $\Lambda_{{\gamma}, \theta}^r$. See Figure~\ref{fig.gamma} for a representation of the rotated setting.

\begin{figure*}[t]
\centering
\includegraphics[width = 15cm]{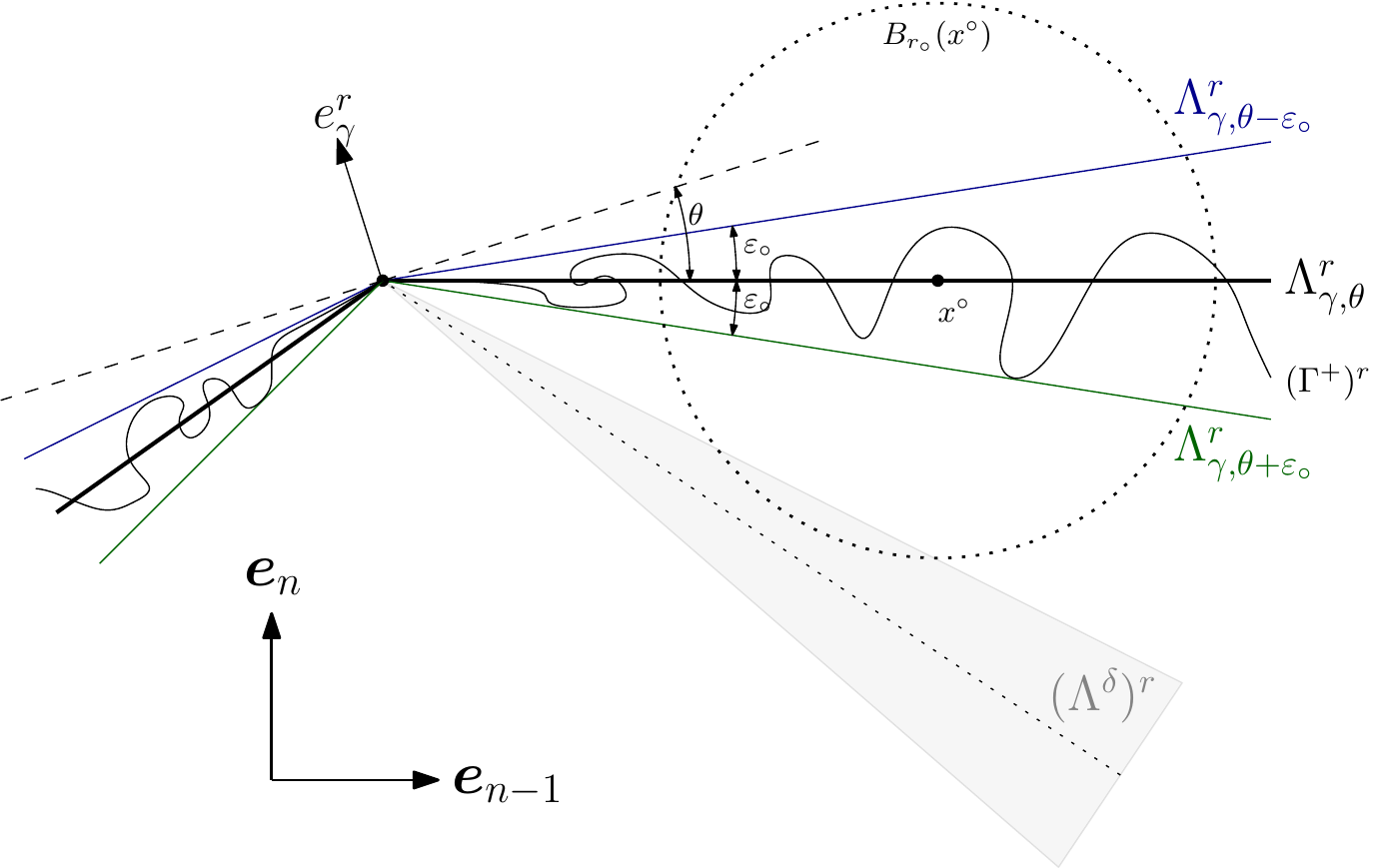}
\caption{Representation of the setting after a rotation.}
\label{fig.gamma}
\end{figure*}

Now take any point $x^\circ \in B_{1/2}\cap \{x_n = 0\}$, so that $x^\circ\in \Lambda_{{\gamma}, \theta}^r$. Denote $r_\circ = x_{n-1}^\circ/2$, and consider a ball $B_{r_\circ}(x^\circ)$. Notice that 
\[
\{x_n \leq -3\tan(\eps_\circ) r_\circ\}\subset E \subset \{x_n \leq 3\tan(\eps_\circ) r_\circ\},\quad\mbox{ in } B_{r_\circ}(x^\circ).
\]

Thus, if $\eps_\circ$ is small enough, we can apply Theorem~\ref{thm.giuch8} rescaled in the ball $B_{r_\circ}(x^\circ)$; which tells us that $(\Gamma^{+})^r$ in $B_{r_\circ}(x^\circ)$ is the graph of a function in the $\boldsymbol{e}_n$ direction. Since we can cover all of $(\Gamma^{+})^r$ with balls of this kind, we conclude that $(\Gamma^{+})^r$ is the graph of a function in the $\boldsymbol{e}_n$ direction in $B_{1/2}\cap \{x_{n-1}\geq 0\}$.

The case $\Phi \not\equiv{\rm id}$ is a perturbation of the previous one, but we would need to use Theorem~\ref{thm.flat_obst_p} instead of Theorem~\ref{thm.giuch8}, since it is no longer true that we are necessarily a minimal surface in $B_{r_\circ}(x^\circ)$.
\\[0.2cm]
{\bf Step 2: $C^{1,1-}$ regularity of $\Gamma_\pm$.} Let us first discuss the case $\Phi\equiv {\rm id}$. In this situation, using \eqref{eq.gamma0}, we obtain that $\Gamma^+$ is a graph that is Lipschitz up to its boundary $\{x_{n-1}=x_n =0\}$ and we may now consider the reflection 
$\Gamma^+_*$ of $\Gamma^+$ under the transformation $(x'', x_{x-1},x_n) \mapsto (x'', -x_{n-1}, -x_n)$. Since $\Gamma^+$ is a Lipschitz graph up to $\{x_{n-1}=x_n =0\}$ the ``odd reflection'' $\overline{\Gamma^+\cup \Gamma^+_*}$
is a Lipschitz graph which solves the equation of minimal graphs in the viscosity sense. It follows that $\overline{\Gamma^+\cup \Gamma^+_*}$ is analytic. 

In the case $\Phi\not\equiv{\rm id}$ we cannot use the reflection trick and the interior smoothness of minimal graph to conclude, but still using \eqref{eq.gamma0} and that $\Phi\in C^{1,1}$ we see that $\Gamma^+$ is a Lipschitz graph with now $C^{1,1}$ boundary datum solving a thick obstacle problem with the mean curvature operator $H$. It follows from standard perturbative methods and the boundary regularity theory for obstacle problems with elliptic operators 
(see, for instance, Jensen \cite{Jen80}) that the $\Gamma^+$ is a $C^{1,\beta}$ graph up to its boundary $\Phi( \{x_{n-1}=x_n =0\})$.
\end{proof}

With this, we can proceed and prove Theorem~\ref{thm.main}.

\begin{proof}[Proof of Theorem~\ref{thm.main}]
If $\theta\in (0,C_\circ\eps)$, then we can directly apply Proposition~\ref{prop.impclos}.

On the other hand, if $\theta\in \big[C_\circ\eps,\frac{\pi}{2}\big)$, thanks to Lemmas~\ref{lem.cneps} and \ref{lem.subslem} we have that 
\[
\Phi(\Lambda_{{\gamma},\theta+C_\circ\eps})\subset E \subset \Phi(\Lambda_{{\gamma}, \theta-C_\circ\eps}),\quad\textrm{ in }B_{1/2}.
\]

That is, by rescaling and taking $\eps$ smaller depending only on $n$ if necessary, we have put ourselves in the situation to apply Lemma~\ref{lem.gamma}. We conclude the proof in this case by noticing Remark~\ref{rem.cinf} and that we can take $\rho_\circ = \frac14$.
\end{proof}

\section{Regularity of solutions}
\label{sec.6}
In this section, in order to simplify the computations, we assume $\Phi \equiv {\rm id}$. All statements and proofs are done under this assumption. We leave to the interested reader the standard extension of this results to the cases
$\Phi \in C^{k,\beta}$, $k\ge 2$ and $\beta\in (0,1)$ or $\Phi$ analytic. 

\begin{prop}
\label{prop.wow}
There exists $\eps_\circ$ depending only on $n$ such that the following statement holds:

Let $E\subset \R^n$ satisfying \eqref{eq.TOP} with $0\in \de E$, be such that $E$ is $\eps$-close to $\Lambda_{{\gamma}, \theta}$ in $B_1$, for some $\eps\in(0,\eps_\circ)$. Then, there exists some $\Lambda_{\overline{\gamma}, \overline{\theta}}$ with $\overline{\gamma}$ and $\overline{\theta}$ as in \eqref{eq.eandtheta}, such that for $\alpha\in \left(0, \frac{1}{2}\right)$,
\[
E \mbox{ is } C_\alpha \eps r^{1+\alpha}\mbox{-close to }\Lambda_{\overline{\gamma}, \overline{\theta}}\mbox{ in } B_r,\quad\mbox{ for all } r\in (0, 1/2),
\]
for some constant $C_\alpha$ depending only on $n$ and $\alpha$. 
\end{prop}
\begin{proof}
We will suppose that $\eps > 0$ is sufficiently small so that each of the results used can be applied. 

We begin by noticing that there are two possible scenarios. Either $\theta \geq C_\circ \eps$ or $\theta < C_\circ\eps$, where $C_\circ$ is the constant given in Lemma~\ref{lem.cneps} and in Proposition~\ref{prop.impclos}, depending only on $n$. 

Notice that if $\theta \geq C_\circ \eps$ we are already done. Indeed, in this case we can apply Lemma~\ref{lem.cneps} and Lemma~\ref{lem.subslem} to fulfill the hypotheses of Lemma~\ref{lem.gamma}; which at the same time yields the desired result, thanks to Remark~\ref{rem.cinf}.

Suppose otherwise that $\theta < C_\circ\eps$. In this case we can apply the improvement of closeness in Proposition~\ref{prop.impclos}. That is, there exist some radius $\rho_\circ$, depending only on $n$ and $\alpha$, such that 
 \[
   E \mbox{ is } \rho_\circ^{1+\alpha} \eps\mbox{-close to } \Lambda_{{\gamma}_2,\theta_2}\mbox{ in } B_{\rho_\circ},
 \]
for some ${\gamma}_2$ and $\theta_2$ as in \eqref{eq.eandtheta}. Let us define $E_2 := \rho_\circ^{-1}E $, so that we have a set $E_2\subset\R^n$, satisfying \eqref{eq.TOP}, with $0\in \de E_2$ and $\rho_\circ^{\alpha}\eps$-close to $\Lambda_{{\gamma}_2, \theta_2}$ in $B_1$. We are now again presented with a dichotomy: either $\theta_2 \geq C_\circ \rho_\circ^{\alpha}\eps$ or $\theta_2 \leq C_\circ \rho_\circ^{\alpha}\eps$. In the former case, we can again apply Lemma~\ref{lem.gamma} and Remark~\ref{rem.cinf} to find that 
\[
E_2 \mbox{ is } C\eps \rho_\circ^\alpha r^{1+\alpha}\mbox{-close to }\Lambda_{\bar{\gamma_2}, \bar\theta_2}\mbox{ in } B_r,\quad\mbox{ for all } r\in (0, 1/2),
\]
for some $\Lambda_{\bar \gamma_2, \bar \theta_2}$ (which is close to $\Lambda_{{\gamma}_2, \theta_2}$). Rescaling back, $E$ is $C\eps r^{1+\alpha}$-close to $\Lambda_{\bar{\gamma_2}, \bar \theta_2}$ in $B_{r}$ for all $r \in (0, \rho_\circ/2)$. Using that $E$ is $\eps$-close to $\Lambda_{{\gamma}, \theta}$ in $B_1$ it follows that $E$ is $C_\alpha\eps r^{1+\alpha}$ close to $\Lambda_{ \bar \gamma_2, \bar \theta_2}$ in $B_r$, for all $r\in (0, 1/2)$, and a constant $C_\alpha$ that depends on $\alpha$ and $n$, of the form $C_\alpha = C\rho_\circ^{-1-\alpha}$ for $C$ depending only on $n$. 

If $\theta_2 \leq C_\circ\rho_\circ^\alpha \eps$, we can repeat the process iteratively. Suppose that for all $k < k_\circ\in \N$, we have $\theta_k \leq C_\circ \rho_\circ^{k\alpha} \eps$, but $\theta_{k_\circ} \geq C_\circ \rho_\circ^{k_\circ\alpha} \eps$. That is, there exist $E_k:= \rho_\circ^{-k+1} E$, satisfying \eqref{eq.TOP}, with $0\in \de E_k$ such that it is $\rho_\circ^{\alpha (k-1)}\eps$-close to $\Lambda_{ \gamma_k, \theta_k}$ in $B_1$. By Lemma~\ref{lem.gamma} and Remark~\ref{rem.cinf},
\begin{equation}
\label{eq.k0}
E_{k_\circ} \mbox{ is } C\eps \rho_\circ^{(k_\circ-1)\alpha} r^{1+\alpha}\mbox{-close to }\Lambda_{\bar{\gamma}_{k_\circ}, \bar\theta_{k_\circ}}\mbox{ in } B_r,\quad\mbox{ for all } r\in (0, 1/2), 
\end{equation}
for some $\Lambda_{\bar \gamma_{k_\circ}, \bar\theta_{k_\circ}}$ (close to $\Lambda_{{\gamma}_{k_\circ}, \theta_{k_\circ}}$) and for some constant $C$ depending only on $n$. Alternatively, we can write
\[
E \mbox{ is } C\eps r^{1+\alpha}\mbox{-close to }\Lambda_{\bar {\gamma}_{k_\circ}, \bar \theta_{k_\circ}}\mbox{ in } B_r,\quad\mbox{ for all } r\in (0, \rho_\circ^{k-1}/2). 
\]

Let us redefine, from now on, and for convenience in the upcoming notation, $\Lambda_{{\gamma}_{k_\circ}, \theta_{k_\circ}} := \Lambda_{\bar \gamma_{k_\circ}, \bar \theta_{k_\circ}}$. Notice that $E_k$ is $\rho_\circ^{\alpha(k-1)}\eps$-close to $\Lambda_{{\gamma}_k, \theta_k}$ in $B_1$, but it is also $\rho_\circ^{\alpha(k-2)-1}\eps$-close to $\Lambda_{{\gamma}_{k-1}, \theta_{k-1}}$. Therefore, 
\begin{equation}
\label{eq.telim}
|\theta_{k}-\theta_{k-1}|+|{\gamma}_{k}-{\gamma}_{k-1}| \leq C_\circ\rho_\circ^{\alpha(k-2)}\left(\rho_\circ^{-1}+\rho_\circ^\alpha\right)\eps = C_{n, \alpha}\rho_\circ^{\alpha k} \eps,
\end{equation}
where the sub-indices denote the only dependences of the constants. In particular, by triangular inequality
\begin{equation}
\label{eq.series}
|\theta_{k_\circ} - \theta_k | +|{\gamma}_{k_\circ} - {\gamma}_k | \leq C_{n, \alpha}\eps \sum_{j = k+1}^{k_\circ} \rho_\circ^{\alpha j} \leq C_{n, \alpha}\eps \frac{\rho_\circ^{\alpha(k+1)}}{1-\rho_\circ^\alpha} = C_{n, \alpha}\eps \rho_\circ^{\alpha k},
\end{equation}
for a different constant $C_{n, \alpha}$, still depending only on $n$ and $\alpha$. Thus, since $E_k$ is $\rho_\circ^{\alpha(k-1)}\eps$-close to $\Lambda_{{\gamma}_k, \theta_k}$ in $B_1$, $E$ is $\rho_\circ^{(1+\alpha)(k-1)}\eps$-close to $\Lambda_{{\gamma}_k, \theta_k}$ in $B_{\rho_\circ^{k-1}}$. 

Now, from \eqref{eq.series}, $\Lambda_{{\gamma}_k, \theta_k}$ is $C_{n, \alpha}\eps\rho_\circ^{\alpha k}\rho_\circ^{k-1}$-close to $\Lambda_{{\gamma}_{k_\circ}, \theta_{k_\circ}}$ in $B_{\rho_\circ^{k-1}}$. Putting all together, $E$ is $C_{n, \alpha}\rho_\circ^{(1+\alpha) (k-1)}\eps$-close to $\Lambda_{{\gamma}_{k_\circ}, \theta_{k_\circ}}$ in $B_{\rho_\circ^{k-1}}$ for all $k < k_\circ$. This, combined with \eqref{eq.k0}, yields the desired result. 

Finally, if $\theta_k \leq C_\circ\rho^{k\alpha}\eps$ for all $k \in \N$, we can take $k_\circ = \infty$ and repeat the previous procedure. In this case, consider as ${e}_\infty$ and $\theta_\infty$ the limits of the sequences $({e}_k)_{k\in \N}$ and $(\theta_k)_{k\in \N}$, which exist by \eqref{eq.telim}. Notice that $\theta_\infty = 0$.
\end{proof}

\begin{rem}
\label{rem.fb}
In the previous proof, notice that if $k_\circ<\infty$ we must be dealing with a point in the interior of the contact set. In particular, all points on the free boundary must have $k_\circ = \infty$, and since $\theta_\infty = 0$ there is a supporting plane at each of this points.
\end{rem}

We now give a proposition on regularity of $\partial E$ in the case that it is close enough to some $\Lambda_{\gamma, \theta}$ with $\theta$ small enough (the wedge is almost a half-space). 

\begin{prop}
\label{thm.thmu}
There exists $\eps_\circ$ depending only on $n$ such that the following statement holds: 

Let $E\subset \R^n$ satisfying \eqref{eq.TOP}, be such that $E$ is $\eps$-close to $\Lambda_{\gamma, \theta}$ in $B_1$, for $\eps\in (0, \eps_\circ)$, and $\theta \leq C_\circ \eps$ for a constant $C_\circ$ depending only on $n$. Then, after a rotation of angle $\gamma$, $\de E$ is the graph of a function $h:B_{1/2}'\to (-1, 1)$ in the $\boldsymbol{e}_n$ direction in $B_{1/2}$. Moreover,
\begin{equation}
\label{eq.ureg}
\|h \|_{C^{1, \alpha}(\overline{B_{1/2}'\cap\{x_{n-1}\geq 0\}})}+\|h \|_{C^{1, \alpha}(\overline{B_{1/2}'\cap\{x_{n-1}\leq 0\}})}\leq C \eps,
\end{equation}
for any $\alpha \in \left(0, \frac{1}{2}\right)$, and some constant $C$ depending only on $n$ and $\alpha$.
\end{prop}
\begin{proof}
Let assume for simplicity that $\gamma = 0$, the other cases are analogous. We will assume that $\eps_\circ$ is small enough so that the previous results can be applied. Let us also assume that the contact set, $\Delta_E := \de E \cap \{x_{n-1} = x_n = 0\}$, is non-empty in $B_{1/2}$; $\Delta_E \cap B_{1/2} \neq \varnothing$. Otherwise we are already done by the classical improvement of flatness. 
\\[0.2cm]
{\bf Step 1: $\de E$ is the graph of a function.} Let us first show that indeed $\de E$ is the graph of a function. To do so, proceed as in the first part of Lemma~\ref{lem.gamma}, combined with Proposition~\ref{prop.wow} and the fact that $\theta \leq C_\circ\eps$:

Take any $x_\circ\in B_{1/2}\cap \de E$ not belonging to the contact set $\Delta_E$, and let $r := \mbox{dist}(x_\circ, \Delta_E) = |x_\circ - z|$ for $z\in \Delta_E$. Applying Proposition~\ref{prop.wow} around $z$, we deduce that for some $\Lambda_{\bar \gamma, \bar \theta}$ (depending on $z$),
\[
E \mbox{ is } C\eps r\mbox{-close to } \Lambda_{\bar \gamma, \bar\theta},\quad\mbox{ in } B_{r/2}(x_\circ),
\]
for some constant $C$ depending only on $n$. If we rescale the space a factor $2r^{-1}$ with respect to $z$ so that $E$ becomes $\tilde E$ then
\[
\tilde E \mbox{ is } C\eps \mbox{-close to } \Lambda_{\bar \gamma, \bar\theta},\quad\mbox{ in } B_{1}(2r^{-1}\,x_\circ).
\]

Notice that $\tilde E$ is a minimal surface in $B_{1}(2r^{-1}\,x_\circ)$, since $E$ is a minimal surface in $B_{r/2}(x_\circ)$. Using that $|\bar \gamma - 0| + |\theta - \bar \theta|\leq C\eps$ for some $C$ depending only on $n$, and that $\theta \leq C_\circ \eps$, we get that $\Lambda_{\bar \gamma, \bar\theta}$ is $C\eps r$-close to $\{x_n = 0\}$ in $B_{r/2}(x_\circ)$. After the rescaling, $\Lambda_{\bar \gamma, \bar \theta}$ is $C\eps$-close to $\{x_n = 0\}$ in $B_1(2r^{-1}\, x_\circ)$, so that $\tilde E$ is $C\eps$-close to $\{x_n = 0\}$ in $B_1(2r^{-1}\, x_\circ)$. Thanks to the classical improvement of flatness (Theorem~\ref{thm.giuch8}) for $\eps$ small enough depending only on $n$, $\de\tilde E$ is a graph in the $\boldsymbol{e}_n$ direction in $B_1(2r^{-1}\, x_\circ)$, and consequently the same occurs for $\de E$ in $B_{r/2}(x_\circ)$. Let us call $h$ the function whose graph is defined on $B_{r/2}(x_\circ)$ in the $\boldsymbol{e}_n$ direction. In particular, applying Theorem~\ref{thm.giuch8} again, $h\in {\rm Lip}(B_{r/4}'(x_\circ'))$, with $[h]_{C^{0,1}(B_{r/2}')}\leq C\eps$; where $x_\circ'$ is the projection of $x_\circ$ to $\{x_n = 0\}$. 

Now, by a standard covering argument together with the fact that $\de E$ is continuous and $\Delta_E$ has measure zero, $u$ is defined in $B_{1/2}'$ with 
\[
[h]_{C^{0,1}(B_{1/2}')}\leq C\eps,
\]
for some $C$ depending only on $n$.
\\[0.2cm]
{\bf Step 2: Regularity bound.} Let us now show \eqref{eq.ureg}. We will show that for any $y' \in B_{1/4}'\cap \{x_{n-1}\geq 0\}$ and any $\rho \in (0, 1/4)$, there exists some $p_{y'}\in \R^{n-1}$ depending only on $y'$ such that for any $\alpha \in (0, 1/2)$,
\begin{equation}
\label{eq.up}
|h(x') - h(y') - p_{y'}\cdot (x'-y')| \leq C \eps \rho^{1+\alpha}\quad\mbox{ in } B_\rho'(y')\cap \{x_{n-1}'\geq 0\},
\end{equation}
for some constant $C$ depending only on $n$ and $\alpha$. The other half, $\{x_{n-1}'\leq 0\}$, follows by symmetry. 

Throughout this second step we will be switching between the characterisation of the solution to our thin obstacle problem as a boundary, $\de E$, and as the graph of a function $u$ on $\R^{n-1}$. Thus, we can rewrite Proposition~\ref{prop.wow}. That is, if $0\in \de E$, we know that 
\begin{equation}
\label{eq.Eu}
E \mbox{ is } C_\alpha \eps r^{1+\alpha}\mbox{-close to }\Lambda_{\overline{\gamma}, \overline{\theta}}\mbox{ in } B_r,\quad\mbox{ for all } r\in (0, 1/2),
\end{equation}
for some constant $C_\alpha$ depending only on $n$ and $\alpha$, and for some $\Lambda_{\bar \gamma, \bar \theta}$. We want to rewrite it in terms of $u$. Note that $|\gamma| +\bar \theta\leq C\eps$ for some constant $C$ depending only on $n$, since $\theta \leq C_\circ\eps$, and therefore, we have that \eqref{eq.Eu} implies
\begin{equation}
\label{eq.Eu2}
|h(x')-A^+(x_{n-1}')_+-A^-(x_{n-1}')_-|\leq C_\alpha \eps |x'|^{1+\alpha},\quad\mbox{ in } B_{1/2}',
\end{equation}
with $A^- \geq A^+$ and $|A^-|+|A^+|\leq C\eps$ for some $C_\alpha$ depending only on $n$ and $\alpha$. Notice that if $0$ is in the free boundary of the contact set, $0\in \de \Delta_E'$, then $A^+ = A^-$, or equivalently $\bar \theta = 0$ (see Remark~\ref{rem.fb}).

Let $y', z' \in B_{1/4}'\cap \{x_{n-1}'\geq 0\}$, and let $y'_\circ, z'_\circ\in \Delta_E'$ be such that ${\rm dist}(y', \Delta_E') = |y'-y_\circ'|$ and ${\rm dist}(z', \Delta_E') = |z'-z_\circ'|$. We denote by $y, z, y_\circ$, and $z_\circ$, the corresponding elements as seen in $\R^n$ (e.g. $y = (y', 0)$), and let $\bar y = (y', h(y'))\in \de E$ and $\bar z = (z', h(z'))\in \de E$. Suppose, without loss of generality, that $d = |y'-y_\circ'|\leq |z'-z_\circ'|$, and we consider two different cases. 
\begin{itemize}
\item {\it Case 1.} Suppose that $r = |z'-y'|\geq d/2$. Using \eqref{eq.Eu2} centered around $y_\circ'$instead of 0, we know that for some $A^+$ depending on $y_\circ'$, 
\[
|h(x')-A^+ x_{n-1}'|\leq C_\alpha \eps |x'-y_\circ'|^{1+\alpha},\quad \mbox{ for } x'\in B_{1/2}'(y_\circ')\cap \{x_{n-1}'\geq 0\}.
\]
Putting $y'$ and $z'$ in the previous expression yields
\begin{align*}
|h(y')-A^+ y_{n-1}'|& \leq C_\alpha \eps|y'-y_\circ'|^{1+\alpha} = d^{1+\alpha} \leq C_\alpha \eps r^{1+\alpha},\\
|h(z')-A^+ z_{n-1}'|& \leq C_\alpha \eps|z'-y_\circ'|^{1+\alpha} \leq C_\alpha \eps\left(d+ r\right)^{1+\alpha} \leq C_\alpha \eps r^{1+\alpha} ,
\end{align*}
from which 
\[
|h(y')- h(z') - A^+(y_{n-1}'-z_{n-1}')| \leq C_\alpha \eps r^{1+\alpha},
\]
and in particular, \eqref{eq.up} holds with $p_{y'} = A^+$. 
\item {\it Case 2.} Suppose $r = |z'-y'| \leq d/2$. If $B_{d}'(y')\nsubseteq \{x_{n-1}'\geq 0\}$, then $y_\circ'\in \Delta_E'$ belongs to the free boundary and the corresponding $\Lambda_{{\gamma}(y_\circ'), \theta(y_\circ')}$ from Proposition~\ref{prop.wow} around $y_\circ$ is actually an hyperplane ($\theta(y_\circ')= 0$) with normal vector $e_{\gamma(y_\circ')}$ (see Remark~\ref{rem.fb}). In particular, $\de E$ is $C\eps d^{1+\alpha}$-flat in the $e_{\gamma(y_\circ')}$ direction in the ball $B_{d}(y)$ thanks to Proposition~\ref{prop.wow}. On the other hand, if $B_{d}'(y')\subset \{x_{n-1}'\geq 0\}$, we consider again the corresponding $\Lambda_{{\gamma}(y_\circ'), \theta(y_\circ')}$ from Proposition~\ref{prop.wow} around $y_\circ$. Then $\de E$ is $C\eps d^{1+\alpha}$-flat in the $e_{\gamma(y_\circ')+\theta(y_\circ')}$ direction in the ball $B_{d}(y)$ (recall that $e_{\gamma(y_\circ')+\theta(y_\circ')}$ is the normal vector to $\Lambda_{{\gamma}(y_\circ'), \theta(y_\circ')}$ in $\{x_{n-1}\geq 0\}$). In any case, noting that $E$ is a set of minimal perimeter in $B_d(y)$ we can apply the classical improvement of flatness (see Remark~\ref{rem.impflat}) in $B_d(y)$, to get
\[
|\nu(y)-\nu(z)|\leq C\eps|y-z|^\alpha,
\]
for some $C$ depending only on $n$. We have denoted here by $\nu(x)$ for $x\in \de E$ the unit normal vector to $\de E$ pointed outwards with respect to $E$ at the point $x$.

Now notice that if $\eps$ is small enough depending only on $n$, since $|\nabla h|\leq C\eps$, $|\nu(y)-\nu(z)| \geq |\nabla h(y')-\nabla h(z')|$, and on the other hand, $|y-z|\leq |y'-z'|+|h(y')-h(z')|\leq 2|y'-z'|$ so that
\[
|\nabla h(y')-\nabla h(z')|\leq C\eps|y'-z'|^\alpha,
\]
from which \eqref{eq.up} follows.
\end{itemize}
From \eqref{eq.up} the result \eqref{eq.ureg} follows by a covering argument. 
\end{proof}

With this, we can now prove Theorem~\ref{thm.main2}. 

\begin{proof}[Proof of Theorem~\ref{thm.main2}]
In the case $\Phi \equiv {\rm id}$ it is a direct consequence of Lemma~\ref{lem.gamma} and Proposition~\ref{thm.thmu}, depending on whether the wedge $\Lambda_{\gamma, \theta}$ is $\eps$-flat or not. The case $\Phi \not\equiv {\rm id}$ follows from standard perturbative arguments and is left to the interested reader.
\end{proof}

\section{Monotonicity formula and blow-ups}
\label{sec.7}
In this section we prove Proposition~\ref{prop.classif_cones} and Corollary~\ref{cor.cor_2}. 

\begin{lem}[Monotonicity formula for minimizers of \eqref{eq.TOP}] 
\label{lem.mono_2}
Let $E\subset \R^n$ satisfy \eqref{eq.TOP} in $B_2$ (instead of $B_1$) and suppose $0 \in \partial E\cap\partial \mathcal{O}$. Let us define 
\begin{equation}
\label{eq.monformphi_2}
\mathcal A(r) := \frac{P\big(E; \Phi(B_r)\big)}{r^{n-1}},\quad\textrm{for}\quad r>0.
\end{equation}
Then, 

(a) If $\Phi\equiv{\rm id}$ then $ \mathcal A'(1) \ge 0$

(b) If $\Phi(0) = 0$, $D\Phi(0) = {\rm id}$, and $[\Phi]_{C^{1,1}} \le \eta_\circ$ for $\eta_\circ\in (0,1)$ small enough depending only on $n$ then
\[ \mathcal A' (1) \ge - C\eta_\circ\]
for some $C$ depending only on $n$.
\end{lem}

\begin{proof}
(a) The proof is similar to that of the classical monotonicity formula for minimal surfaces. Indeed, we  take as a competitor to $E$ in $B_{1}$ the dilation of $E$ to $B_{1-\eps}$ and we extend it conically in the annulus. For simplicity in the following computations, from now on we rescale everything by a factor 2, so that we can deal with $r = 1$ and $ \mathcal{A}'(1)$.

As in \cite{Sav10b}, we take $F$ defined as
\begin{equation}
\label{eq.original}
x\in F \Leftrightarrow  \left\{ \begin{array}{ll}
  x \in E & \textrm{if }|x| > 1\\
  x/|x| \in E& \textrm{if } (1-\eps)\le |x|\le 1\\
  (1-\eps)^{-1} x \in E & \textrm{if }|x|< (1-\eps),\\
  \end{array}\right.
\end{equation}
that is, we first contract it by a factor $1-\eps$ and then extend conically $F$ in the annulus $B_1\setminus B_{1-\eps}$ to obtain a competitor for $E$ in $B_{1}$.

Thus, 
\begin{equation}
\label{eq.comp}
P_{B_{1}}(E) \le P_{B_{1}}(F) = (1-\eps)^{n-1} P_{B_{1}}(E) + P_{B_{1}\setminus B_{1-\eps}} (F).
\end{equation}
Now, dividing by $\eps$ and letting $\eps\downarrow 0$, we obtain
\begin{equation}
\label{eq.comp2}
(n-1)P_{B_{1}}(E) \le \mathcal{H}^{n-2}(\partial E\cap \partial B_{1}).
\end{equation}

On the other hand, notice that 
\begin{equation}
\label{eq.conecond}
\mathcal{A}'(1) = \int \frac{1}{\sqrt{1-(x\cdot \nu(x))^2}} d\mathcal{H}^{n-2}_{\de E\cap \de B_{1}} - (n-1) P_{B_{1}} (E),
\end{equation}
which combined with \eqref{eq.comp2} yields the result in the case (a). 

(b) The proof in this case is a perturbation of the proof in case (a). Now we have 
\[
\Phi(0) =0, \quad D\Phi(0) \equiv {\rm id} \quad \mbox{and} \quad |D^2\Phi| \le \eta_\circ \quad \mbox{in }B_1,
\]

The observation that allows us to control the errors is that, for all  $x_\circ \in B_1$.
\begin{equation}\label{eq.11111}
\Phi(x) =  \Phi(x_\circ) + D\Phi(x_\circ) (x-x_\circ) + O( \eta_\circ |x-x_\circ|^2),
\end{equation}
\begin{equation}\label{eq.22222}
D\Phi(x_\circ) = {\rm id} + O(\eta_\circ) , \quad D\Phi(rx_\circ) = D\Phi(x_\circ)+ O(\eta_\circ(1-r)), \quad \forall r\in (0,1).
\end{equation}

As a consequence, for $r \in (0,1]$ the maps $\theta : (0,1]\times\Phi(B_1) \rightarrow \Phi(B_r)$ defined by
\[
(r,x) \mapsto \Phi\big( r\Phi^{-1}(x)\big)  
\]
are bi-Lipschitz and are quasi-dilations with the estimate, for $r\in (1/2,1)$
\begin{equation}\label{eq.errorest}
|\theta(r,x) - \theta(r,x_\circ)| \le  r|x-x_\circ| \big(1 + C(1-r)\eta_\circ\big).
\end{equation}
Indeed, \eqref{eq.errorest} follows immediately from \eqref{eq.11111} and \eqref{eq.22222} if $|x_\circ-x|<(1-r)$. For general $x_\circ, x$ we use the previous case and the triangle inequality.

Now, repeat the proof for the case (a) after applying $\Phi^{-1}$ and then check using \eqref{eq.errorest} that the errors we make are small. Namely, we define $F$ as in \eqref{eq.original} but with $E$ replaced by $\Phi^{-1}(E)$. Note that $\Phi(F)$ is a ``competitor'' of $E$ in $\Phi(B_1)$, namely, $\Phi(\Lambda^\delta) \subset \Phi(F)$ and $\Phi(F)\setminus \Phi(B_1)= 
E\setminus \Phi(B_1)$.
 
Now \eqref{eq.comp} must be replaced by
\begin{equation}
\label{eq.comp0}
P_{\Phi(B_{1})}(E) \le P_{\Phi(B_{1})}(\Phi (F)) = P_{\Phi(B_{1-\eps})}( \Phi(F)) + P_{\Phi( B_{1}\setminus B_{1-\eps}) } (\Phi(F)).
\end{equation}
Now, using \eqref{eq.errorest} and $\Phi(F) = \theta(1-\eps, E)$ in $\Phi(B_{1-\eps})$, we obtain 
\[
 P_{\Phi(B_{1-\eps})}(\Phi(F))  \le   (1-\eps)^{n-1} P_{\Phi(B_{1})}(E) + O(\eta_\circ \eps).
\]
and 
\[
 P_{\Phi(B_1\setminus B_{1-\eps})}(\Phi(F)) = \eps \mathcal{H}^{n-2} \big( \Phi(F\cap \partial B_1)\big) + O(\eta_\circ \eps).
\]

So that, 
\[
P_{\Phi(B_{1})}(E) \le (1-\eps)^{n-1}P_{\Phi(B_{1})}(E) + \eps \mathcal{H}^{n-2} \big( \Phi(F\cap \partial B_1)\big) + O(\eta_\circ \eps).
\]

Dividing by $\eps$ and letting $\eps\downarrow 0$ we obtain 
\[
(n-1)P_{\Phi(B_{1})}(E) \le \mathcal{H}^{n-2}(\partial E\cap \Phi(\partial B_{1}))+O(\eta_\circ).
\]

Now we conclude the proof observing that 
\[
\mathcal{A}'(1) = \int \frac{\big|\partial_r \theta\big(1,  \Phi^{-1}(x)\big)\big|}{\sqrt{1-(x\cdot \nu(x))^2}} d\mathcal{H}^{n-2}_{\de E\cap \Phi(\de B_{1})} - (n-1) P_{\Phi(B_{1})} (E),
\]
and that $\big|\partial_r \theta\big(1,  \Phi^{-1}(x)\big)\big| = 1 + O(\eta_\circ)$.
\end{proof}

\begin{lem}[Monotonicity formula for minimizers of \eqref{eq.TOP}] 
\label{lem.mono}
Let $E\subset \R^n$ satisfy \eqref{eq.TOP} and suppose $0 \in \partial E\cap\partial \mathcal{O}$. Let us define 
\begin{equation}
\label{eq.monformphi}
\mathcal A_E(r) := \frac{P\big(E; \Phi(B_r)\big)}{r^{n-1}},\quad\textrm{for}\quad r>0.
\end{equation}
Then, 

(a) If $\Phi\equiv{\rm id}$ then $\mathcal A '\ge 0$ for $r\in(0,1)$. Moreover, $\mathcal{A}' \equiv 0$ (i.e., $\mathcal{A}$ constant) if and only if $E$ is a cone ($tE = E$ for any $t > 0$).

(b) If $\Phi(0) = 0$, $D\Phi(0) = {\rm id}$, and $[\Phi]_{C^{1,1}} \le \eta_\circ$ for $\eta_\circ\in (0,1)$ small enough depending only on $n$ then
\[
\mathcal{A}_E'(r) \ge - C\eta_\circ
\] 
for some $C$ depending only on $n$.

\end{lem}
\begin{proof}
It follows by scaling Lemma \ref{lem.mono_2}. Part (a) is immediate, being the cone condition an immediate consequence of \eqref{eq.conecond}. For part (b), let us define, for any $\lambda > 0$, $\Phi^\lambda:= \lambda\Phi\left(\frac{1}{\lambda}\,\cdot\,\right)$, and 
\begin{equation}
\label{eq.not.lem}
\mathcal{A}^\lambda_E(r) := \frac{P\big(E; \Phi^\lambda(B_r)\big)}{r^{n-1}},\quad\textrm{for}\quad r>0.
\end{equation}

Note now, that 
\[
\mathcal{A}_E(r) = \frac{P\big(\lambda E; \lambda \Phi(B_r)\big)}{\lambda^{n-1}r^{n-1}} =\frac{P\big(\lambda E; \Phi^\lambda(B_{\lambda r})\big)}{\lambda^{n-1}r^{n-1}}=\mathcal{A}_{\lambda E}^\lambda(\lambda r).
\]
Differentiating both sides with respect to $r$ we obtain 
\begin{equation}
\label{eq.mono0}
\mathcal{A}_E'(r) = \lambda \left(\mathcal{A}^\lambda_{\lambda E}\right)'(\lambda r).
\end{equation}

On the other hand, applying Lemma~\ref{lem.mono_2} with $\lambda E$ and $\Phi^\lambda$,
\[
 \left(\mathcal{A}^\lambda_{\lambda E}\right)'(1) \ge -C[\Phi^\lambda]_{C^{1,1}(B_1)} \ge  -C \lambda^{-1} \eta_\circ.
\]
Putting it together with \eqref{eq.mono0} and fixing $\lambda = r^{-1}$ we obtain 
\[
\mathcal{A}'_E(r) = r^{-1} \left(\mathcal{A}^\lambda_{\lambda E}\right)'(1) \ge -C\eta_\circ,
\]
as we wanted to see.
\end{proof}

We now recall the well-known density estimates lemma for perimeter minimizers. It is a very standard result in the theory of minimal surfaces which can be found extensively in the literature. We mention, for example, the survey \cite{Sav10}. 

\begin{lem}
\label{lem.clasdensest}
Let $E\subset\R^n$ be a minimizer of the perimeter in $B_{r_\circ}$ for some $r_\circ > 0$, such that $0 \in \de E$. Then, 
\begin{align*}
|E\cap B_r| &\geq c r^{n},\\
 |E^c\cap B_r| &\geq cr^{n},\quad\textrm{for all }\quad r \in (0, r_\circ),
\end{align*}
for some $c$ constant depending only on the dimension $n$.
\end{lem}

We have a similar lemma for supersolutions to the minimal perimeter problem. 

\begin{lem}
\label{lem.densest}
Let $E^+\subset\R^n$ be a supersolution to the minimal perimeter problem in $B_{r_\circ}$ for some $r_\circ > 0$, such that $0 \in \de E^+$. Then, 
\begin{align*}
 |(E^+)^c\cap B_r| &\geq cr^{n},\quad\textrm{for all }\quad r \in (0, r_\circ),
\end{align*}
for some $c$ constant depending only on the dimension $n$.
\end{lem}
\begin{proof}
This is standard, and follows exactly the same as Lemma~\ref{lem.clasdensest}.
\end{proof}


Let us now prove the following proposition, stating that in order to prove that at some scale the solution is close enough to a wedge, it is enough to classify conical solutions.

\begin{prop}
\label{prop.rhoE_P}
Assume that in some dimension $n\ge 2$ the wedges $\Lambda_{{\gamma},\theta}$ are the only cones $E\subset \R^n$ satisfying \eqref{eq.TOP} with $\Phi = {\rm id}$ and any $\delta > 0$.

Assume that, for some $\delta>0$, the set $E\subset \R^n$ with $P(E;B_1)<\infty$ satisfies $\Phi(\Lambda^{\delta}) \cap B_1\subset E$ and \eqref{eq.TOP}, with $\Phi$ a $C^{1,1}$ diffeomorphism.

Then, for any $\eps > 0$, there exists $\rho > 0$ depending only on $n$, $\eps$, and $\|\Phi\|_{C^{1,1}}$, and $\|D\Phi^{-1}\|_{L^\infty}$, such that if $x_\circ \in \de E\cap \de\mathcal O \cap \overline{B_{1/2}}$, then
\[
\rho^{-1} (R_{x_\circ} E-x_\circ)\quad\textrm{ is }\eps\textrm{-close to } \Lambda_{{\gamma}, \theta},
\]
for some $\gamma$ and $\theta$ as in \eqref{eq.eandtheta} and for some rotation $R_{x_\circ}$ depending only on $x_\circ$.
\end{prop}

\begin{proof}
After a translation, let us start by assuming that $x_\circ = 0$. Let us also take a rotation $R_{x_\circ}$ of the whole setting, in such a way that, if we denote $\Phi_k := k\Phi$, then $R_{x_\circ} \Phi_k(\Lambda^\delta)$ converges in Hausdorff distance locally to $\Lambda^{\delta'}$ as $k\to \infty$ for some $\delta' > 0$ (i.e., we take the blow-up of a Lipschitz boundary). Notice that the value $\delta'$ is determined only by $\delta$ and $\Phi$. By redefining $\Phi$ if necessary, let us assume $R_{x_\circ} = {\rm id}$ for simplicity. (Note that we could also argue via Lemma~\ref{lem.asumpPhi}.)

Let us argue by contradiction, and assume that the thesis does not hold.

Let $\rho_k = k^{-1}$, and consider the sequence of sets $E_k = \rho_k^{-1}E$. Notice that, for $\Phi_k := k\Phi$, each $E_k$ fulfils $\Phi_k(\Lambda^\delta) \cap B_k\subset E_k$ and solves a thin obstacle problem of the type 
\begin{equation}\label{eq.TOPk}
 P(E_k;B_k)\le P(F;B_k) \quad \forall F\  \mbox{such that }E_k\setminus B_k = F\setminus B_k \mbox{ and } \Phi_k(\Lambda^{\delta}) \cap B_k\subset F.
\end{equation}
Recall that the set $\Phi_k(\Lambda^\delta)$ converges in Hausdorff distance to $\Lambda^{\delta'}$ as $k\to \infty$. From minimality, we have compactness in $L^1_{\rm loc}$ of $E_k$, so that, up to a subsequence, $E_k\xrightarrow{L^1_{\rm loc}}E_\infty$, for some global solution to the $\delta'$-thin obstacle problem with $\Phi = {\rm id}$, $E_\infty$, with $\Lambda^{\delta'} \subset E_\infty$. It immediately follows that $0\in \overline{E_\infty}$. 

On the other hand, by the density estimates in Lemma~\ref{lem.densest}, since each $E_k$ is a supersolution to the minimal perimeter problem in $B_1$ and $0 \in \de E_k$ for all $k$, we  have 
\[
|E_k^c\cap B_r|\ge c r^n,\quad\textrm{ for all } r \in (0, 1),
\]
for some constant $c$. The convergence in $L^1_{\rm loc}$ implies that the limit also fulfils $|E_\infty^c\cap B_r| \ge cr^n$, and therefore $0\in \de E_\infty$.

Using the same notation as in the proof of Lemma~\ref{lem.mono} (see \eqref{eq.not.lem}), we know 
\[
\mathcal{A}_E(r) = \mathcal{A}^k_{E_k}(kr),\quad\textrm{for all } r > 0.
\]

Notice, also, that
\[
\mathcal{A}_{E_k}^k(r)  \to \mathcal{A}_{E_\infty}(r) := \frac{P\big(E; B_r\big)}{r^{n-1}}\quad\textrm{locally as }k\to \infty,
\]
where we are using the $L^1_{\rm loc}$ convergence of $E_k$ to $E_\infty$, and the fact that $\Phi^k = k\Phi(k^{-1}\,\cdot\,)\to {\rm id}$ as $k\to \infty$ in $C^{1,1}_{\rm loc}$. In particular, we have that 
\[
 \lim_{\rho \downarrow 0} \mathcal{A}_E(\rho) = \mathcal{A}_{E_\infty} (r),\quad\textrm{ for all } r > 0.
\]
Thanks to Lemma~\ref{lem.mono} part (b), the left-hand side limit is well defined. That is, $\mathcal{A}_{E_\infty}(r)$ is bounded and constant for any $r > 0$, which, from Lemma~\ref{lem.mono} part (a) implies that $E_\infty$ is a cone ($tE_\infty = E_\infty$ for any $t > 0$). By assumption, therefore, $E_\infty = \Lambda_{\gamma, \theta}$ for some $\gamma$ and $\theta$; and we have that $E_k$ is converging in $L^1_{\rm loc}$ to some $\Lambda_{\gamma, \theta}$. 

Finally, in order to reach the contradiction, let us show that the convergence of $\de E_k$ to $\de E_\infty$ is in Hausdorff distance locally, which will complete the proof. 

Suppose that is is not. That is, after extracting a subsequence, we can assume that there exists some sequence of points $y_k\in \de E_k$ such that $y_k \to y_\infty$ and ${\rm dist}(y_k, \de E_\infty) > \eps > 0$ for some $\eps > 0$ and for all $1\le k\le\infty$. We have a dichotomy, either $y_\infty \in E_\infty$ or $y_\infty \in E_\infty^c$.

Let us now use the density estimate in Lemma~\ref{lem.densest}. If $y_\infty\in E_\infty$ then, after a subsequence if necessary, $|E_k^c\cap B_\eps(y_k)|\ge c \eps^n$ but $|E_\infty^c\cap B_\eps(y_\infty)| = 0$, which is a contradiction with the $L^1_{\rm loc}$ convergence. On the other hand, if $y_\infty\in E_\infty^c$ assume that after a subsequence $y_k\in E_\infty^c$ for all $k > 0$. We have that for $k$ large enough $y_k\in \de E_k$ is a point around which $E_k$ is a minimal surface (being $E_\infty$ a barrier \emph{from below}). That is, we can use the classical density estimates for minimal surfaces in Lemma~\ref{lem.clasdensest} to reach that $|E_k\cap B_\eps(y_k)|\ge c\eps^n$ but $|E_\infty \cap B_\eps(y_\infty)| = 0$, again, a contradiction. 
\end{proof}

 Thus, in order to prove Corollary~\ref{cor.cor_2}, it will be enough to classify cones.

\begin{proof}[Proof of Proposition~\ref{prop.classif_cones}]
The proof is by induction on the dimension $n$. 
\\[0.1cm]
{\bf Step 1: Base case. Dimension $n = 2$.} 

Assume that $\Sigma^2\subset \R^2$ is a cone satisfying \eqref{eq.TOP}, in other words, the boundary of $\Sigma^2$ in $B_1$ consists of radii of length one. By assumption, we have $(0, -1)\in\Sigma^2\cap S^1$. Now, if $\Sigma^2$ were not a wedge (that is, if $\Sigma^2\cap S^1$ were disconnected) then the convex hull of $\Sigma^2\cap B_1$ would be a set containing the obstacle (it contains $\Sigma^2$) and having strictly less relative perimeter in $B_1$ than $\Sigma^2$. This would contradict the minimality of $\Sigma^2$ ---i.e. \eqref{eq.TOP}. 
\\[0.1cm]
{\bf Step 2: Induction step.} Suppose that it holds up to dimension $n-1 \ge 2$. Let us show it for dimension $n$. 


Let us first prove regularity of the cone around contact points. Assume that we have, without loss of generality, $x_\circ = \boldsymbol{e}_1 = (1,0,\dots,0)\in \de\Sigma\cap\de B_1$. The first thing to notice is that the blow up of $\Sigma$ around $x_\circ$ is a wedge $\Lambda_{\gamma_1,\theta_1}$. Indeed, the blow-up is a cone by the monotonicity formula, and thanks to the fact that $\Sigma$ is a cone and $x_\circ = \boldsymbol{e}_1$, we get that the blow up at $x_\circ$ must be of the form $\R\times\Sigma^{n-1}$; where now $\Sigma^{n-1}\subset\R^{n-1}$ is a cone in $n-1$ dimensions such that satisfies \eqref{eq.TOP} (also taking $\Lambda^\delta$ in $n-1$ dimensions). In particular, by induction step, $\Sigma^{n-1} = \Lambda_{\gamma_1,\theta_1}^{n-1}\subset \R^{n-1}$, where $\Lambda_{\gamma_1,\theta_1}^{n-1}$ denotes $\Lambda_{\gamma_1,\theta_1}$ as seen in $n-1$ dimensions. This immediately yields that the blow up at $x_\circ$ is a wedge of the form $\Lambda_{\gamma_1,\theta_1}$. By Proposition~\ref{prop.rhoE_P} and Theorem~\ref{thm.main2}, $\de \Sigma$ is a smooth minimal surface around any $x_\circ\in \de\Sigma\cap\{x_{n-1} = x_n = 0\}$ in $\{\pm x_{n-1}\ge 0\}$ up to $\{x_{n-1} = 0\}$. 

Let us separate the proof between both sides $\pm x_{n-1} \ge 0$, and let us focus first on $x_{n-1} \ge 0$ (the other side follows analogously). We can now take $s^* = \max\{s\ge \delta : \Lambda^s \subset \Sigma\textrm{ in }x_{n-1} \ge 0\}$. Notice that it is indeed a maximum, since it is enough to check that $\Lambda^s\cap S^{n-1}\subset \Sigma\cap S^{n-1}$, where $S^{n-1}\subset \R^n$ denotes the $(n-1)$-dimensional sphere. 

The boundaries $\de\Sigma\cap S^{n-1}$ and $\de\Lambda^{s^*} \cap S^{n-1}$ must  touch at a point $x_\circ \in \{x_{n-1} \ge 0 \}$. If $x_\circ\in \{x_{n-1} > 0\}$, then by the strong maximum principle for minimal surfaces we must have $\Sigma_\mathcal{O} = \Lambda^{s^*}$ in $\{x_{n-1}\ge 0\}$, where $\Sigma_{\mathcal{O}}$ denotes the connected component of $\Sigma\setminus\{x_{n-1} = x_n = 0\}$ that contains the thin obstacle $\mathcal{O}$ (which, in this case, is flat). On the other hand, if $x_\circ \in \{x_{n-1} = x_n = 0\}$, then we have previously shown (by induction and dimension reduction) that $\de\Sigma\cap \{x_{n-1}\ge0\}$ is $C^1$ up to its boundary around the points $x_\circ$ and touches the half-plane of $\de\Lambda^{s^*}$ tangentially at $x_\circ$. Using the boundary strong maximum principle (Hopf lemma) we obtain again that $\Sigma_\mathcal{O} = \Lambda^{s^*}$ in $\{x_{n-1}\ge 0\}$.

The same holds for the other side, $x_{n-1}\le 0$, so that in all we have that
\[
\Sigma_{\mathcal{O}} = \Lambda_{\gamma,\theta}
\]
for some $\gamma$ and $\theta$ as in \eqref{eq.eandtheta}. 

We can now repeat the argument, but opening $\Lambda_{\gamma,\theta}$ instead, until we reach another connected component of $\Sigma\setminus\{x_{n-1} = x_n = 0\}$.
Proceeding iteratively, this yields that $\Sigma$ must be one dimensional; that is, $\Sigma$ is the cone $\R^{n-2}\times\Sigma^{2}$ for some cone $\Sigma^2\subset \R^2$.
 By the base case in Step 1 minimality implies that $\Sigma^2$ must be a convex angle and hence $\R^{n-2}\times\Sigma^{2}$ is a wedge.
\end{proof}

Once cones are classified, we can proceed with the proof of Corollary~\ref{cor.cor_2},

\begin{proof}[Proof of Corollary~\ref{cor.cor_2}]
We will apply Theorem~\ref{thm.main2} after an translation, rotation, and scaling. We have to check that the hypotheses are fulfilled. 


By definition of minimizer of \eqref{eq.thepb} (see Definition \ref{defi.notionmin})   there exist $\delta_k\downarrow 0$, $E_k$ minimizers of 
\eqref{eq.wedgepb}
such that $\chi_{E_k}\rightarrow \chi_E$ in $L^1(B_1)$. 
For each $E_k$ let $x_\circ$ be any point in $B_{1/2}\cap\partial E_k \cap \partial \mathcal O$.
Let $E^{x_\circ, \rho}_k := \psi_{x_\circ}(E_k) = \rho^{-1} (R_{x_\circ} E_k-x_\circ)$, where $\psi_{x_\circ}$ denotes the change of coordinates from Lemma~\ref{lem.asumpPhi}. Let us also denote $\Phi_\rho^{x_\circ} := \bar \Phi$ the new diffeomorphism (also from Lemma~\ref{lem.asumpPhi}). 

Thus, $E^{x_\circ, \rho}_k$ is a minimizer of the $\bar\delta$-thin obstacle problem around $x_\circ$ with diffeomorphism $ \Phi_{x_\circ}^\rho$ such that $ \Phi_\rho^{x_\circ} (0) = 0$, $D \Phi_\rho^{x_\circ} (0) = {\rm id}$, and $[ \Phi_\rho^{x_\circ}]_{C^{1,1}(B_1)} \le  C\rho$ thanks to Lemma~\ref{lem.asumpPhi}. 

On the other hand, as a consequence of Proposition~\ref{prop.classif_cones} and  Proposition~\ref{prop.rhoE_P} in any dimension $n \ge 2$, we reach that, for $\rho$ small enough,  $E^{x_\circ, \rho}_k$ is  $\eps_\circ$-close to $\Lambda_{\gamma, \theta}$ 
for some $\gamma$ and $\theta$. Also, for $\rho$ small enough, we will have 
$[ \Phi_\rho^{x_\circ}]_{C^{1,1}(B_1)} \le \eps_\circ^{1+\frac{1}{2}}$ where $\eps_\circ>0$ is the constant in Theorem \ref{thm.main2}. Therefore, applying Theorem \ref{thm.main2} to $E^{x_\circ, \rho}_k$  (and shrinking by a factor $\rho$) we obtain that
 $\de E_k$ has the following $C^{1,\alpha}$ structure in $B_{\rho/2}(x_\circ)$. Either:
\begin{enumerate}
\item[(a)] In appropriate coordinates $y$,  $(\Phi^{x_\circ})^{-1}\big(R_{x_\circ}(\partial E_k-x_\circ) \big)$ is the graph $\{y_n = h(y')\}$ of a function $h\in C^0(\overline{B_{\rho/2}'})$ satisfying $h \in C^{1,\alpha}(\overline{B_{\rho/2}'^+})\cap C^{1,\alpha}(\overline{B_{\rho/2}'^-})$. Moreover, we have $h\ge 0$ on $y_{n-1} = 0$ and $\nabla h$ is continuous on $\{y_{n-1} = 0\}\cap \{h > 0\}$. 
\end{enumerate}
\noindent or 
\begin{enumerate}
\item[(b)]   $R(\partial E_k-x_\circ) \cap B_{\rho/2}$ is the union of two $C^{1,1-}$ surfaces that meet on $\partial \mathcal O$ with full contact set in $B_{\rho/2}$.
\end{enumerate}

Now we deduce in case (a) that in some new coordinates with origin at $x_\circ$ we have
$\Phi^{-1}\big(\partial E_k \big)$ is the graph $\{z_n = \tilde h(z')\}$ of a function $\tilde h\in C^0(\overline{B_{\tilde \rho}'})$ satisfying $\tilde h \in C^{1,\alpha}(\overline{B_{\tilde \tilde \rho}'^+})\cap C^{1,\alpha}(\overline{B_{\tilde \rho}'^-})$. Moreover, we have $\tilde h\ge 0$ on $z_{n-1} = 0$ and $\nabla \tilde h$ is continuous on $\{z_{n-1} = 0\}\cap \{\tilde h > 0\}$.

Since either (a) or (b) holds for $E_k$ with estimates independent of $k$, we can pass to the limit and show that  either (a) or (b) also holds for $E$.

Finally, if the alternative (b) near some point $x_\circ$ then  using that $\partial \mathcal O$ is of class $C^{k,\beta}$ (and the classical $C^{k,\beta}$ regularity up to the boundary results for minimal surfaces \cite{Jen80}) we obtain that $\partial E$ splits into two $C^{k,\beta}$ minimal surfaces with boundary in a small ball around $x_\circ$.
\end{proof}

\begin{proof}[Proof of Theorem~\ref{thm.generic}]
After having introduced the appropriate notion of solution, we have that Theorem~\ref{thm.generic} corresponds to Corollary~\ref{cor.cor_2}. 
\end{proof}

\end{document}